\documentclass[11pt,leqno]{amsart} 
\usepackage[margin=1.2in]{geometry} 
\usepackage{kbpn}    
 \usepackage{xypic} 
 \usepackage{pdflscape}
 \usepackage{adjustbox}
 
\title[Syntomic cohomology of $\BP\langle n\rangle$]{Syntomic cohomology of truncated\\ 
 Brown--Peterson spectra}
\author[G.~Angelini-Knoll]{Gabriel Angelini-Knoll}
\address{Department of Mathematics, Applied Mathematics and Statistics, Case Western Reserve University, Cleveland, OH, USA}
\email{gabriel.angelini-knoll@case.edu}
      
\begin{document}

\begin{abstract}     
We compute the $\MU$-based syntomic cohomologies, mod $(p,v_1,\cdots,v_n)$, of all $\bE_1$ $\MU$-algebra forms of the truncated Brown--Peterson spectrum $\BP\langle n\rangle$. 
As qualitative consequences, we resolve the Lichtenbaum--Quillen, telescope, and redshift questions for the algebraic K-theories of all $\bE_1$ $\MU$-algebra forms of $\BP \langle n\rangle$. 
This extends work of Hahn and Wilson. 
We also explicitly compute the mod $(p,v_1,v_2)$ algebraic K-theory of arbitrary $\bE_1$ $\MU$-algebra forms of $\BP\langle 2\rangle$ at all primes $p\ge 5$ 
extending previous work of the author, Ausoni, Culver, H\"oning, and Rognes.  
Additionally, we present a new computation of mod $(p,v_1,v_2,v_3)$ algebraic K-theory of arbitrary $\bE_1$ $\MU$-algebra forms of $\BP\langle 3\rangle$ at all primes $p\ge 7$, the first explicit computation of algebraic K-theory of an $\mathbb{E}_1$-ring of height $3$. 
\end{abstract}  
           
\maketitle                 
\tableofcontents       
          
\section{Introduction}
By celebrated work of Quillen~\cite{Qui69}, complex cobordism provides a bridge between geometric topology and algebraic geometry. 
Working $p$-locally, a homology theory called Brown--Peterson homology $\mathrm{BP}$ splits off of complex cobordism. 
This homology theory has coefficients a polynomial algebra over the $p$-local integers with generators $v_i$ for $i\ge 1$. 
Complex oriented homology theories carry a formal group law, and the complex oriented Brown--Peterson homology theory carries the universal $p$-typical (1-dimensional) formal group law. 
By Lazard~\cite{Laz55}, $p$-typical formal group laws are classified by their height and we may ask whether there is a complex oriented homology theory that carries a formal group law of height $\le n$. 
One way to build such a ring spectrum is to take a quasi-regular quotient of Brown--Peterson homology to produce truncated Brown--Peterson homology $\mathrm{BP}\langle n\rangle$ 
with coefficients a polynomial algebra over the $p$-local integers on generators $v_1$, $v_2$, $\dots$ , $v_n$. By convention, we set $\mathrm{BP}\langle -1\rangle=\mathbb{F}_p$. 

Structured ring spectra play a fundamental role in homotopy theory, for example more highly structured ring spectra are equipped with a richer notion of power operations. 
It is known that complex cobordism is an $\bE_\infty$-ring, which means it is a homology theory equipped with a coherent commutative multiplication. It was hoped that the truncated Brown--Peterson homology theories would have similar structure. 
It was shown that $\bE_\infty$-ring structures exist on $\BP\langle n\rangle$ for $n=-1,0,1$ at all primes and for $n=2$ at the primes $p=2,3$ by~\cite{BR08,HL10,LN12,LN14}. However, it was proven by Lawson and Senger~\cite{Law18,Sen24} that $\BP\langle n\rangle$ do not have $\bE_{\infty}$-ring structures when $n\ge 4$, so we do not live in the best of all possible worlds.

Producing $\bE_m$-ring structures on truncated Brown--Peterson spectra for $m>1$ at arbitrary heights $n$ remained an open problem until Hahn--Wilson~\cite{HW22} proved 
that there exist $\bE_3$ $\MU$-algebra forms of truncated Brown--Peterson spectra for all primes $p$ and heights $n$. These forms depend on a choice of generators $v_i$. The most common choice would be to choose the Araki or Hazewinkel generators. 
However, it was observed by Strickland~\cite[Remark~6.5]{Str99} at the prime $2$ that, although these forms of truncated Brown--Peterson spectra are equipped with $\bE_1$-ring structures, they are not equipped with additional structure. In particular, they are not equipped with $\bE_2$-ring structures.

To compute invariants of \emph{all} $\bE_1$ $\MU$-algebra forms of truncated Brown--Peterson spectra we are therefore forced to work without additional structure. 
The algebraic K-theory and topological Hochschild homology of an $\bE_1$ $\MU$-algebra is only an $\bE_0$ algebra over the corresponding invariant for $\MU$, where an $\bE_0$ algebra simply consists of a unit map.
Nevertheless, the fact that the unit map $\MU_{(p)}\to \BP\langle n\rangle$ is a $\pi_*$-surjection allows us to proceed. We therefore resolve the redshift, Lichtenbaum--Quillen and telescope questions in the case of algebraic K-theory of arbitrary $\mathbb{E}_1$-$\mathrm{MU}$-algebra forms of truncated Brown--Peterson spectra, see~\cite[\S~1.1]{AKHW24} for details. 
Let $T(n)=\mathbb{S}/(p^{i_0},v_1^{i_1}\cdots ,v_{n-1}^{i_{n-1}})[v_n^{-1}]$ for suitably large integers $i_0,i_1,\cdots ,i_{n-1}$. We write $L_n$ for the Bousfield localization at $\mathbb{Q}\oplus K(1)\oplus \cdots \oplus K(n)$ where $K(i)$ is the $i$-th Morava K-theory at the prime $p$. We also write $L_n^f$ for the Bousfield localization at $\mathbb{Q}\oplus T(1)\oplus \cdots \oplus T(n)$. 
\begin{thmx}[\autoref{cor:LQC-K-theory}, \autoref{cor:telescope} and \autoref{cor:redshift}]\label{thm:main-qualitative-thm}
Let $p$ be a prime number and $n\ge -1$ be an integer. Let $\BP\langle n\rangle$ be an arbitrary $\bE_1$ $\MU$-algebra form of the $n$-th truncated Brown--Peterson spectra at the prime $p$, see \autoref{form def}. Then 
\begin{enumerate}
\item \label{it:redshift} (Redshift) The spectrum $T(n+1)\otimes \K(\BP\langle n\rangle)$ is nonzero. 
\item  \label{it:telescope} (Telescope) The  localization map $L_{n+1}^f\K(\BP\langle n\rangle)\to L_{n+1}\K(\BP\langle n\rangle)$ is an equivalence.  
\item \label{LQ} (Lichtenbaum--Quillen) The localization map $\K(\BP\langle n\rangle)_{(p)}\longrightarrow L_{n+1}^f\K(\BP\langle n\rangle)_{(p)}$ has bounded above fiber.  
\end{enumerate}
\end{thmx}
\begin{remark}
When $n=-1,0,1$, \autoref{thm:main-qualitative-thm}~\eqref{it:redshift},~\eqref{LQ} were first proven by \cite{Qui72}, \cite{BM94}, and \cite{AR02} at the primes $p\ge 2$, $p\ge 3$, and $p\ge 5$ respectively. For arbitrary $\bE_3$ $\MU$-algebra forms of truncated Brown--Peterson spectra Hahn--Wilson proved \autoref{thm:main-qualitative-thm}~\eqref{it:redshift},~\eqref{LQ}  in~\cite{HW22}. The cases $n=-1,0,1$ of \autoref{thm:main-qualitative-thm}~\eqref{it:telescope} were already known by \cite{Mah82}, \cite{Mil81}, and \cite{HRW22}. 
Our theorem extends these results to arbitrary $\bE_1$ $\MU$-algebra forms of truncated Brown--Peterson spectra, such as those arising from the choice of Araki or Hazewinkle generators, at arbitrary heights $n$ and primes $p$.  
\end{remark}

\autoref{thm:main-qualitative-thm} follows from the following explicit computation. 
\begin{thmx}[\autoref{thm:syntomic-BPn} and \autoref{mod-(p ,...,v_n) syntomic}]\label{thm:main-computation}
Let $n\ge -1$ be an integer and let $\langle n+1\rangle=\{1,2,\cdots ,n+1\}$. Let $\BP\langle n\rangle$ be an arbitrary $\bE_1$ $\MU$-algebra form of the $n$-th truncated Brown--Peterson spectra at the prime $p$, see \autoref{form def}.
We can identify the mod $(p,v_1,\cdots ,v_{n})$ syntomic cohomology of $\BP\langle n\rangle$ with
\[
\mathbb{F}_p[v_{n+1}]\otimes \left ( \bF_p\langle \partial,\lambda_1,\lambda_2,\cdots,\lambda_{n+1}\rangle \oplus \bigoplus_{j=1}^{n+1}\bF_p\langle \lambda_s :s\in \langle n+1\rangle -\{i\}\rangle\{\Xi_{j,d} :0<d<p\} \right )
\]
as a bigraded $\bF_p[v_{n+1}]$-modules. Here the classes are in bidegrees  $\|v_{n+1}\|=(2p^{n+1}-2,0)$, $\|\partial\|=(-1,1)$, $\|\lambda_i\|=(2p^i-1,1)$,  and $\|\Xi_{j,d}\|=(2p^j-1-2dp^{j-1},1)$ for each $1\le i,j\le n+1$ and $0<d<p$. 
\end{thmx}
Here and throughout we write $R\langle x_1,x_2,\cdots ,x_n\rangle$ for an exterior algebra over a ring $R$ with generators $x_1,x_2,\cdots ,x_n$ and $R[y_1,y_2,\cdots ,y_m]$ for a polynomial algebra over a ring $R$ with generators $y_1,y_2,\cdots ,y_m$. 

\begin{remark}
When $n=-1,0,1$, \autoref{thm:main-computation} was already known by~\cite{BMS19},~\cite{LW22} and~\cite{HRW22} respectively. 
\end{remark}

\begin{figure}[!ht]
\resizebox{\textwidth}{!}{ 
\begin{tikzpicture}[radius=1,xscale=1.3,yscale=2]
\foreach \n in {-2,0,...,26} \node [below] at (\n,-.8-1) {$\n$};
\foreach \s in {-1,0,...,5} \node [left] at (-.3-2,\s) {$\s$};
\draw [thin,color=lightgray] (-2,-1) grid (26,5);
\node [below] at (0,0) {$1$};

\node [below] at (-1,1) {$\partial$};
\node [below] at (1,1) {$\Xi_{1,1}$};
\node [above] at (3,1) {$\lambda_1$};
\node [below] at (3,1) {$\Xi_{2,1}$};
\node [above] at (7,1) {$\lambda_2$};
\node [below] at (7,1) {$\Xi_{3,1}$};
\node [below] at (15,1) {$\lambda_3$};

\node [below] at (2,2) {$\partial\lambda_1$};
\node [above] at (6,2) {$\partial\lambda_2$};
\node [below] at (6,2) {$\lambda_1\Xi_{2,1}$};
\node [below] at (8,2) {$\Xi_{1,1}\lambda_2$};
\node [above] at (10,2) {$\lambda_1\lambda_2$};
\node [below] at (10,2) {$\lambda_1\Xi_{3,1}$};
\node [below] at (14,2) {$\partial\lambda_3$};
\node [above] at (14,2) {$\lambda_2\Xi_{3,1}$};
\node [above] at (16,2) {$\Xi_{1,1}\lambda_3$};
\node [below] at (18,2) {$\lambda_1\lambda_3$};
\node [above] at (18,2) {$\Xi_{2,1}\lambda_3$};
\node [above] at (22,2) {$\lambda_2\lambda_3$};

\node [above] at (9,3) {$\partial\lambda_1\lambda_2$};
\node [above] at (17,3) {$\partial\lambda_1\lambda_3$};
\node [below] at (17,3) {$\lambda_1\lambda_2\Xi_{3,1}$};
\node [above] at (21,3) {$\lambda_1\Xi_{2,1}\lambda_3$};
\node [below] at (21,3) {$\partial\lambda_2\lambda_3$};
\node [above] at (23,3) {$\Xi_{1,1}\lambda_2\lambda_3$};
\node [above] at (25,3) {$\lambda_1\lambda_2\lambda_3$};

\node [above] at (24,4) {$\partial\lambda_1\lambda_2\lambda_3$};
\end{tikzpicture}
}
\caption{The mod $(2,v_1,v_2,v_3)$-syntomic cohomology of $\BP\langle 2\rangle$. 
\label{fig:syntomicBP2}
}
\end{figure}

\begin{remark}
The syntomic cohomology computed in \autoref{thm:main-computation} and depicted in \autoref{fig:syntomicBP2} at $n=p=2$ suggests a certain rotational symmetry akin to the Lagrangian refinement of Tate duality of~\cite{Bha}. 
A higher height analogue of this duality is the subject of ongoing work of Devalapurkar--Hahn--Rognes, see~\cite{Rog25}.
\end{remark}

When $n=2$ and $p\ge 5$, we argue that the motivic spectral sequence collapses and we can prove the following. 

\begin{thmx}[\autoref{TCBP2} and \autoref{KBP2}]\label{thm:TC}
Let $p\ge 5$ and let $\BP\langle 2\rangle$ be an arbitrary $\bE_1$ $\MU$-algebra form of the second truncated Brown--Peterson spectrum at the prime $p$, see \autoref{form def}. Let $\langle 3\rangle=\{1,2,3\}$. We can identify $\pi_*(\TC(\BP\langle 2\rangle)/(p,v_1,v_2))$ with 
\[
\bF_p[v_3] \otimes \left ( \mathbb{F}_p\langle \partial,\lambda_1,\lambda_2,\lambda_{3}\rangle \oplus \bigoplus_{j=1}^{3}\mathbb{F}_p\langle \lambda_s :s\in \langle 3\rangle-\{j\}\rangle\{\Xi_{j,d} :0<d<p\} \right )
\]
where $|v_3|=2p^3-2$, $|\partial|=-1$, $|\lambda_j|=2p^i-1$ and $|\Xi_{j,d}|=2p^j-1-2dp^{j-1}$ for each $1\le i, j\le 3$ and $0<d<p$. There is a long exact sequence 
\[ 
0 \to \Sigma^{-2}M_2\to \pi_*(\K(\BP\langle 2\rangle)/(p,v_1,v_2))\to \pi_*(\TC(\BP\langle 2\rangle)/(p,v_1,v_2))\to \Sigma^{-1}\bF_p\{\partial\}\to 0
\]
where  $M_2=\bF_p\{ \overline{\tau}_1,\overline{\tau}_2,\overline{\tau}_1\overline{\tau}_2\}$ and $|\tau_i|=2p^i-1$ for $i=1,2$. 
\end{thmx}

\begin{remark}
When $n=2$, \autoref{thm:TC} was first proven by~\cite{AKACHR25} at primes $p\ge 7$ for arbitrary $\bE_{3}$ $\MU$-algebra forms. We extend this result to the prime $p=5$ and to arbitrary $\bE_{1}$ $\MU$-algebra forms.
\end{remark}

We can also show that the motivic spectral sequence collapses when $n=3$ and $p\ge 7$. 

\begin{thmx}[\autoref{TCBP3} and \autoref{KBP3}]\label{thm:TCBP3}
Let $p\ge 7$ and let $\BP\langle 3 \rangle$ be an arbitrary $\bE_1$ $\MU$-algebra form of the third truncated Brown--Peterson spectrum at the prime $p$, see \autoref{form def}. Let $\langle 4\rangle =\{1,2,3,4\}$. We can identify $\pi_*(\TC(\BP\langle 3\rangle)/(p,v_1,v_2,v_3))$ with
\[
\bF_p[v_4]\otimes \left ( \mathbb{F}_p\langle \partial,\lambda_1,\lambda_2,\lambda_{3},\lambda_4\rangle \oplus \bigoplus_{j=1}^{4}\bF_p\langle \lambda_s :s\in \langle 4\rangle-\{j\}\rangle\{\Xi_{j,d} :0<d<p\} \right )
\]
where $|v_4|=2p^4-2$, $|\partial|=-1$, $|\lambda_i|=2p^i-1$ and $|\Xi_{j,d}|=2p^j-1-2dp^{j-1}$ for each $1\le i,j\le 4$ and $0<d<p$. There is a long exact sequence 
\[ 
0 \to \Sigma^{-2}M_3\to \K_*(\BP\langle 3\rangle)/(p,v_1,v_2,v_3)\to \TC_*(\BP\langle 3\rangle)/(p,v_1,v_2,v_3)\to \Sigma^{-1}\bF_p\{\partial\}\to 0
\]
where $M_3=\mathbb{F}_p\{\overline{\tau}_1,\overline{\tau}_2,\overline{\tau}_3,\overline{\tau}_1\overline{\tau}_2,\overline{\tau}_1\overline{\tau}_3,\overline{\tau}_2\overline{\tau}_3,\overline{\tau}_1\overline{\tau}_2\overline{\tau}_3 \}$ and $|\overline{\tau}_i|=2p^i-1$ for $i=1,2,3$. 
\end{thmx}

\begin{remark}
\autoref{thm:TC} and \autoref{thm:TCBP3} solve the ``chromatic redshift problem'' of \cite{Rog00} in the affirmative. In particular, \autoref{thm:TCBP3} is the first example of a resolution of this problem for a $\mathbb{E}_1$-ring of height $3$ in the sense of \autoref{def-height}. 
\end{remark}

\begin{remark}
When $n>3$, the answer for syntomic cohomology is exactly the answer predicted by Ausoni and Rognes~\cite[p.~5]{AR02}, as stated  precisely in~\cite[\S~2]{Rog01}, for topological cyclic homology of $\BP\langle n\rangle$ modulo $(p,v_1,\cdots ,v_n)$. For their prediction to hold on the nose, one would need to show that the motivic spectral sequence collapses at the $\EE_2$-term, which is not immediate. Moreover, the spectrum $\mathbb{S}/(p,v_1,\cdots,v_n)$ doesn't exist at small primes and it is not known whether it exists when $n>3$. 
\end{remark}

\begin{remark}
In fact, using \autoref{thm:main-computation}, T. Yang~\cite[Theorem~2.7]{Yan25} has provided a precise bound in the statement of \autoref{thm:main-qualitative-thm} \eqref{LQ}, in particular the map  
\[ \pi_*\K(\BP\langle n\rangle )_{(p)} \longrightarrow \pi_*L_{n+1}^f\K(\BP\langle n\rangle)_{(p)}\]
is an isomorphism in positive degrees. 
\end{remark}
 
\subsection{Acknowledgements}
The author is indebted to Jeremy Hahn and Dylan Wilson who contributed significantly to this project, but declined to be co-authors. 
Additionally, the author benefited from input from Christian Ausoni, H. \"Ozg\"ur Bayindir, Robert Burklund, Ishan Levy, Piotr Pstr\k{a}gowski, Arpon Raksit, John Rognes, Andrew Senger, and Tristan Yang. The author would also like to thank Rin Ray and Shay Ben-Moshe for helpful feedback on a previous draft. The author is also grateful to the Max Planck Institute for Mathematics in Bonn for its hospitality and financial support.  

\section{Hochschild homology}\label{sec:HochschildBPn}
First, we define forms of $\BP\langle n\rangle$ in \autoref{sec:forms} and the $\MU$-based motivic filtration in \autoref{sec:mot}. 
We compute Hochschild homology of $\BP \langle n \rangle$ with $\bF_{p}$-coefficients in \autoref{sec:HH} along with its motivic filtrations in \autoref{sec:motHH}. 

\subsection{Forms of \texorpdfstring{$\BP\langle n\rangle$}{BPn}}\label{sec:forms}
By~\cite{BM13}, there is an $\bE_4$-ring unit map
\[
\BP \to \MU_{(p)}\,.
\]
While $\pi_*\BP$ is only non-canonically isomorphic to a polynomial ring $\mathbb{Z}_{(p)}[v_1,v_2,\cdots]$, for each $n \ge 1$ the subring $\mathbb{Z}_{(p)}[v_1,v_2,\cdots,v_n] \subset \pi_*\BP$ is well-defined. 
Indeed, it is the subring generated by all elements of degree at most $2p^n-2$.  
Based on this, the following definition is standard:

\begin{defin}\label{form def}
Let $m\ge 0$ be an integer. An \emph{$\bE_m$ $\MU$-algebra form of $\BP \langle n \rangle$} is a $p$-local $\bE_m$ $\MU$-algebra $R$ such that the composite
\[
\mathbb{Z}_{(p)}[v_1,v_2,\cdots,v_n] \subset \BP_* \subset \pi_*\MU_{(p)} \to \pi_*R
\]
is an isomorphism.  The last map in this composite is the $p$-localized unit of the $\bE_m$ $\MU$-algebra structure.
\end{defin}
In particular, the Eilenberg--MacLane spectrum $\mathbb{Z}_{(p)}$ is an $\bE_{\infty}$ $\MU$-algebra form of $\BP \langle 0 \rangle$ and the Adams summand $\ell$ of $p$-local topological K-theory $\mathrm{ku}$ is a $\bE_\infty$ $\MU$-algebra form of $\BP \langle 1 \rangle$ by~\cite[Proposition~6.2.1]{HRW22}. At $p=2$, the spectrum $\mathrm{tmf}_1(3)$ is an $\mathbb{E}_\infty$ $\MU$-algebra form of $\BP \langle 2 \rangle$ by \cite[Theorem~1.7]{Sen23}. By convention, the Eilenberg--MacLane spectrum $\mathbb{F}_p$ is a $\mathbb{E}_\infty$ $\MU$-algebra form of $\BP\langle -1\rangle$. 

In this paper, we will be interested in the algebraic $K$-theories of $\bE_1$ $\MU$-algebra forms of $\BP \langle n \rangle$.  Of course, algebraic $K$-theory depends only on the underlying $\bE_1$ $\mathbb{S}$-algebra structure, but we will exploit $\bE_1$ $\MU$-algebra structure when making our computations.

In the remainder of this paper, we fix particular (but arbitrary) $\bE_1$ $\MU$-algebra forms of $\BP \langle n \rangle$, according to the convention below. Since the choices are arbitrary, the theorems we prove about $\BP \langle n \rangle$ hold for all $\bE_1$ $\MU$-algebra forms.

\begin{convention}
For the remainder of this paper, we use the symbol $\BPn$ to denote a fixed (but arbitrary) $\bE_1$ $\MU$-algebra form of $\BPn$. 

We also fix for each $i \ge 1$ an indecomposable polynomial generator 
\[
    v_i \in \pi_{2p^i-2} \BP\,,
\]
and denote also by $v_i$ the image of this class under the map $\pi_{2p^i-2} \BP \to \pi_{2p^i-2} \MU_{(p)}$.  We make these choices such that, for each $i>n$, the class $v_i$ maps to zero in $\pi_{2p^i-2} \BP \langle n \rangle.$
\end{convention}

\begin{convention}
Throughout the paper, whenever a prime $p$, the notations $T(n)$, $L_n$, $L_n^f$ and $\BP\langle n\rangle$ appear in a statement then the implicit primes are always the same and agree with the given prime $p$ if it appears. 
\end{convention} 
\subsection{The motivic filtration}\label{sec:mot}
We review the definition of the motivic filtration relative to $\MU$, defined on the topological Hochschild homology of any $\bE_1$ $\MU$-algebra by \cite{HRW22}. 

We say a spectrum $X$ is even if $\pi_{2k-1}X=0$ for all integers $k$.   Let $\mathrm{CAlg}$ denote the $\infty$-category of $\mathbb{E}_{\infty}$-rings in spectra, let $\mathrm{CAlg}^{\ev}$ denote the full sub-$\infty$-category of  $\mathrm{CAlg}$ whose underlying spectra are even and let $\mathrm{CAlg}_p^{\ev}$ denote the full sub-$\infty$-category of the $\infty$-category of $\mathbb{E}_{\infty}$-rings that are $p$-complete consisting of those with bounded $p$-power torsion. By a filtered object in an $\infty$-category $\mathcal{C}$, we mean a functor $\mathbb{Z}^{\textup{op}}\to \mathcal{C}$ where $\mathbb{Z}^{\textup{op}}$ is the category whose objects are integers such that $\mathbb{Z}^{\textup{op}}(i,j)=*$ if $i\ge j$ and empty otherwise. By a graded object in an $\infty$-category $\mathcal{C}$, we mean a functor $\mathrm{dis}(\mathbb{Z})\to \mathcal{C}$ from the discrete category $\mathrm{dis}(\mathbb{Z})$ whose objects are integers to $\mathcal{C}$. We write $\mathrm{Fil}(\mathcal{C})$ for the $\infty$-category of filtered objects in $\mathcal{C}$ and $\mathrm{Gr}(\mathcal{C})$ for the $\infty$-category of graded objects in $\mathcal{C}$. Recall that the associated graded functor 
\[ \mathrm{gr}^* : \mathrm{Fil}(\mathcal{C}) \to \mathrm{Gr}(\mathcal{C}) \]
is a symmetric monoidal left adjoint when $\mathcal{C}$ is a presentably symmetric monoidal stable $\infty$-category where $\mathrm{Fil}(\mathcal{C})$ and $\mathrm{Gr}(\mathcal{C})$ are equipped with the Day convolution symmetric monoidal structure. Recall that the $d$-speed Whitehead filtration functor on the $\infty$-category of spectra
\[ \tau_{\ge d*} : \mathrm{Sp}\to \mathrm{Fil}(\mathrm{Sp}) \]
is lax symmetric monoidal for any integer $d\ge 1$ and thefore it restricts to a functor 
\[ \tau_{\ge d*} : \mathrm{Alg}(\mathrm{Sp})\to \mathrm{Alg}(\mathrm{Fil}(\mathrm{Sp})) \]
from the infinity category $\mathrm{Alg}(\mathrm{Sp})$ of $\mathbb{E}_1$-rings to the $\infty$-category of $\mathbb{E}_1$-algebras in filtered spectra.

\begin{defin}[{\cite[Proposition~A.1.2,~Construction~A.1.3, and Remark A.1.4]{HRW22}}]
Let $A$ be an $\mathbb{E}_{\infty}$-ring and let $M$ be an $A$-module. We define the even filtration as the limit  
\[ \fil_{\ev/A}^*M:=\lim_{A\to B,B\in \mathrm{CAlg}^{\ev}} \tau_{\ge 2*}(M\otimes_{A}B)  \,.
\]
This is functorial in pairs $(A,M)$ where $A$ is an $\mathbb{E}_\infty$-ring and $M$ is an $A$-module. When $A=M$, we write $\fil_{\ev}^*A:=\fil_{\ev/A}^*A$ and in this case the functor takes values in the category of $\mathbb{E}_{\infty}$ algebras in filtered spectra. 
In light of this, $\fil_{\ev/A}^*M$ takes values in modules over $\fil_{\ev}^*A$ in filtered spectra. When $M$ is an $\mathbb{E}_m$-ring for $m\ge 0$ it takes values in $\mathbb{E}_m$ $\fil_{\ev}^*A$-algebras in filtered spectra. We write $\gr_{\ev/A}^*M$ for the associated graded of  $\fil_{\ev/A}^*M$ and note that when $M$ is an $\mathbb{E}_m$-ring for $m\ge 0$ it takes values in $\mathbb{E}_m$ $\gr_{\ev}^*A$-algebras in graded spectra. 

\end{defin}

A key fact that we will use throughout is the following result, which crucially builds on~\cite[Corollary~1.2]{GWX21} and~\cite[Theorem~6.12]{GIKR22}. See also~\cite[Theorem~1.2]{Pst23} and \cite[Corollary~3.18]{Pst} for a different perspective. 

\begin{thm}[{\cite[Corollary~1.1.6]{HRW22}}]\label{thm:even-motivic}
We can identify 
\[ 
\mathrm{fil}_{\ev}^*\mathbb{S} \simeq \lim_{\Delta} \tau_{\ge 2*}\left (  {\mathrm{MU}}^{\otimes \bullet+1} \right ) \,.
\]
as $\mathbb{E}_\infty$-algebras in filtered spectra. There is a monoidal equivalence
\[ 
\mathrm{Mod}_{(\mathrm{fil}_{\ev}^*\mathbb{S})_p^{\wedge}}(\mathrm{Fil}(\mathrm{Sp}))\simeq \mathrm{Mod}_{\mathbb{S}_p}(\mathrm{SH}^{\textup{cell}}(\mathbb{C}))
\] 
and a monoidal equivalence 
\[ 
\mathrm{Mod}_{(\mathrm{gr}_{\ev}^*\mathbb{S})_p^{\wedge}}(\mathrm{Fil}(\mathrm{Sp}))\simeq \mathrm{Stab}(\mathrm{CoMod}(\mathrm{BP}_*\mathrm{BP})^{\ev})\,.
\] 
Here $ \mathrm{Stab}(\mathrm{CoMod}(\mathrm{BP}_*\mathrm{BP})^{\ev})$ denotes Hovey's unbounded stable derived category of even $\mathrm{BP}_*\mathrm{BP}$-comodules and $\mathrm{Mod}_{\mathbb{S}_p}(\mathrm{SH}^{\textup{cell}}(\mathbb{C}))$ denotes the stable $\infty$-category of modules over the $p$-complete motivic sphere spectrum in the cellular stable motivic $\infty$-category over $\mathrm{Spec}(\mathbb{C})$. 
\end{thm}

\begin{defin}\label{def: grev Moore spectra}
We write $\mathrm{gr}_{\textup{ev}}^*\mathbb{S}/(p,v_1,\cdots ,v_n)$ for the $\mathbb{E}_\infty$ $(\mathrm{gr}_{\textup{ev}}^*\mathbb{S})_p^{\wedge}$-algebra associated to the commutative algebra $\mathrm{BP}_*/(p,v_1,\cdots ,v_n)$ in even $\mathrm{BP}_*\mathrm{BP}$-comodules using \autoref{thm:even-motivic}. Given a  $\mathrm{gr}_{\textup{ev}}^*\mathbb{S}$-module $X$, we write 
\[X/(p,v_1,\cdots ,v_n)=X\otimes_{\mathrm{gr}_{\textup{ev}}^*\mathbb{S}}\mathrm{gr}_{\textup{ev}}^*\mathbb{S}/(p,v_1,\cdots ,v_n)\] using the relative Day convolution symmetric monoidal structure. 
\end{defin}

\begin{defin}\label{def:grading-convention}
Given a graded spectrum $M^{*}$, we say $x\in \pi_nM^w$ has degree $n$, Adams weight $2w-n$ and weight $w$. We write $\|x\|=(n,2w-n)$ and simply $|x|=n$ in this case. We then write $\Sigma^{a,b}M$ for the bigraded suspension that satisfies $\pi_{s}(\Sigma^{a,b}M)^t= \pi_{s-a}M^{t-(a+b)/2}$
so that if $x\in \pi_{s}(\Sigma^{a,b}M)^t$ then it is in stem $s-a$ and Adams weight $2t-s-b$. When $b=0$, we simply write $\Sigma^{a}M$. 
\end{defin}

\begin{remark}
Note that a consequence of \autoref{thm:even-motivic} is that there is an isomorphism 
\[ \pi_n(\grev^w\bS)_p^{\wedge}\cong \Cotor_{(\BP_*,\BP_*\BP)}^{s,2t}(\BP_*\BP_*) \,. \]
where $2t-s=n$ is the degree, $t=w$ is the weight and $s=2w-n$ is the Adams weight. 
\end{remark}

\begin{remark}\label{rem:iterated-cofiber}
By \cite[Theorem~4.3.2]{Rav86} for example, we know that the graded spectrum~$\pi_*\mathrm{gr}_{\textup{ev}}^*\mathbb{S}/(p,v_1,\cdots ,v_n)$ agrees with $\mathbb{F}_p[v_{n+1}]$ in Adams weight zero where $\|v_{n+1}\|=(2p^{n+1}-2,0)$ and consequently we have a periodic self-map 
\[ v_{n+1} : \Sigma^{2p^{n+1}-2}\mathrm{gr}_{\textup{ev}}^*\mathbb{S}/(p,v_1,\cdots ,v_n) \longrightarrow \mathrm{gr}_{\textup{ev}}^*\mathbb{S}/(p,v_1,\cdots ,v_n) \]
so $\mathrm{gr}_{\textup{ev}}^*\mathbb{S}/(p,v_1,\cdots ,v_n)$ could alternatively be constructed as an iterated quotient. 
\end{remark}

\begin{defin}
Suppose that $R$ is an $\bE_1$ $\MU$-algebra. Then we follow \cite[Definition 4.2.1]{HRW22} in defining
\[
\fil_{\mot/\MU}^* \THH(R) := \fil_{\ev/ \THH(\MU)}^* \THH(R)\,,
\]
which is a filtered $\bE_0$-$\gr^*_{\ev}\THH(\MU)$-algebra.  By~\cite[Example~4.2.3,~Corollary~A.2.6]{HRW22}, 
\[
\fil^*_{\mot/\MU} \THH(R) \simeq \mathrm{Tot}\left( \tau_{\ge 2*} \left ( {\THH(R)}\otimes_{\THH(\MU)}{\MU}^{\otimes_{\THH(\MU)}\bullet+1} \right) \right) \,.
\]
\end{defin}

In this particular paper, we will always study motivic filtrations relative to $\MU$.  Thus, we make the following further simplifying convention:

\begin{convention}
Given a $\THH(\MU)$-module $M$, we abbreviate $\fil^*_{\ev/\THH(\MU)}M$ as $\fil^*_{\mot}M$ throughout. Therefore, if $R$ is an $\bE_1$ $\MU$-algebra, we write $\fil^*_{\mot}\THH(R)$ for $\fil^*_{\mot/\MU}\THH(R)$. 
\end{convention}

\begin{warning} \label{warning:MU}
For a general $\bE_1$ $\MU$-algebra $R$, using $\filmot^* \THH(R)$ to denote the filtered spectrum $\fil^*_{\mot/\MU} \THH(R)$ may be in contradiction 
with the notations of \cite{HRW22,Pst}. However, for the particular $R$ we study in this paper, all potential meanings of $\filmot^* \THH(R)$ are equivalent, c.f.~\cite[Example~6.28]{Pst}. 
\end{warning}

\begin{defin}\label{def:quotient of grev MU-modules}
For each $i>0$, we let $\gr_{\ev}^*\MU/v_{i}$ denote the cofiber of $v_{i}$ considered as a self-map of $\gr^*_{\ev} \MU$. For $M$ a $\gr_{\ev}^*\MU$-module, we then define
\[ M/v_{i}:=M\otimes_{\gr_{\ev}^*\MU} (\gr_{\mot}^*\MU/v_{i}).\]
We sometimes write $M/(p,v_1,\cdots ,v_{i})$ for the iterated relative Day convolution 
\[M\otimes_{\gr_{\ev}^*\MU} (\gr_{\mot}^*\MU/p) \otimes_{\gr_{\ev}^*\MU} \cdots \otimes_{\gr_{\ev}^*\MU} (\gr_{\mot}^*\MU/v_i). \]
\end{defin}

\begin{remark}
Via the natural map $\MU \to \THH(\MU)$, any $\gr^*_{\ev} \THH(\MU)$-module is also a $\gr^*_{\ev} \MU$-module.  Thus, if $R$ is an $\bE_1$ $\MU$-algebra, we may refer to
\[ \left  ( \gr^*_{\mot}\THH(R)  \right ) / v_i \qquad \text{ and } \qquad  \left  ( \gr^*_{\mot}\THH(R)  \right ) / (p,v_1,\cdots,v_i) \,. \]
\end{remark}

We observe that the notations in \autoref{def:quotient of grev MU-modules} and \autoref{def: grev Moore spectra} are compatible.
\begin{remark}\label{all-defin-same}
For each $i \ge 1$, our convention that $v_i \in \pi_{2p^i-2} \MU_{(p)}$ is in the image of the unit map 
\[\pi_{2p^i-2}\BP \to \pi_{2p^i-2}\MU_{(p)}\]
ensures that the natural map
\[\gr^*_{\ev} \mathbb{S} / (p,v_1,\cdots,v_{i-1})  \to \gr^*_{\ev} \MU / (p,v_1,\cdots,v_{i-1})\]
sends $v_i$ to $v_i$.  In particular, if $M$ is a $\gr^*_{\ev} \MU$-module, then
\[M \otimes_{\gr^*_{\ev} \mathbb{S}} (\gr^*_{\ev}\mathbb{S}/(p,v_1,\cdots,v_i)) \cong M \otimes_{\gr^*_{\ev} \MU} (\gr^*_{\ev}\MU / (p,v_1,\cdots,v_i))\]
and the notation $M/(p,v_1,\cdots ,v_i)$ is well-defined. 
Additionally, if the Smith--Toda complex $\mathbb{S}/(p,v_1,\cdots ,v_n)$ exists then by \autoref{rem:iterated-cofiber} and \cite[Corollary~1.1.6]{HRW22}, we have 
\[M \otimes_{\gr^*_{\ev} \mathbb{S}} (\gr^*_{\ev/\mathbb{S}}(\mathbb{S}/(p,v_1,\cdots,v_i)) \cong M \otimes_{\gr^*_{\ev} \MU} (\gr^*_{\ev}\MU / (p,v_1,\cdots,v_i))\]
so each of these notations are consistent whenever they are all well-defined. Here we inductively use the fact the map $\Sigma^{2p^i-2}\mathbb{S}/(p,v_1,\cdots ,v_{i-1})\to \mathbb{S}/(p,v_1,\cdots ,v_{i-1})$ induces a monomorphism after applying $\mathrm{MU}_*$  and apply~\cite[Lemma~4.23]{Pst23}. 
\end{remark}

\subsection{Hochschild homology with \texorpdfstring{$\bF_p$}{Fp}-coefficients}\label{sec:HH}

We begin with non-motivic results. First, recall the known computations of $\THH_*(\MU;\mathbb{F}_p)$ and $\THH_*(\BP;\mathbb{F}_p)$. Here and throughout given an $\bE_1$-ring $R$ and an $R\otimes R^{\op}$-module $M$ we write $\THH(R;M):=M\otimes_{R\otimes R^{\op}}R$. Also, given an $\mathbb{E}_\infty$-ring $A$, an $\mathbb{E}_1$ $A$-algebra $R$ and a $R$-bimodule $M$, we define  
\[\THH(R/A;M):=\mathrm{THH}(R;M)\otimes_{\mathrm{THH}(A;M)}M \,. \]

\begin{rec}\label{rec:THHMU}
There is an isomorphism 
\[
\THH_*(\bF_p)=\mathbb{F}_p[\mu]
\]
of graded $\bF_p$-algebras by~\cite{Bok87a}, where $|\mu|=2$, and an isomorphism 
\[ 
\THH_*(\MU;\bF_p)=\mathbb{F}_p\langle \lambda_i^{\prime}: i\ge 1 \rangle
\]
of graded $\bF_p$-algebras,
 where $|\lambda_i^{\prime}|=2i+1$ by \cite[Remark~4.3]{MS93}. 
 We write 
 $\lambda_j\coloneqq\lambda_{p^j-1}^{\prime}$. The $\bE_3$-ring $\THH(\BP;\bF_p)$ is a retract of $\THH(\MU;\bF_p)$ as an $\mathbb{E}_3$-ring by~\cite{BM13} and has homotopy groups
 \[
 \THH_*(\BP;\bF_p) = \mathbb{F}_p\langle \lambda_i: i \ge 1\rangle \,.
 \]
\end{rec}

The following  result is known for $\bE_3$ $\MU$-algebra forms of $\BP\langle n \rangle$, but is new for arbitrary $\bE_{1}$ $\MU$-algebra forms. 
\begin{proposition}\label{prop:Hocschild-MaySSBPn}
As an $\bE_0$ $\THH_*(\BP)$-algebra,
\[
\THH_*(\BP \langle n \rangle;\bF_p)\cong \bF_p \langle \lambda_1,\cdots,\lambda_{n+1}\rangle [\mu^{p^{n+1}}].
\]
Here, the class $\mu^{p^{n+1}}$ denotes a polynomial generator that maps to the similarly named class in $\THH_*(\bF_p)$.
\end{proposition}

\begin{proof}
Consider the Hochschild--May spectral sequence~\cite{AKS18,Kee25,LL23} with $\EE_1$-page
\[\THH_*(\bF_p[v_0,\cdots,v_n];\bF_p)\] 
as an $\bE_0$ algebra over the Hochschild--May spectral sequence for $\BP$. Explicitly, we know $\BP \langle n \rangle$ is an $\bE_1$ $\BP$-algebra by Chadwick--Mandell \cite[Corollary~1.3,~Theorem~1.2]{CM15} and consequently, we know that the filtered  spectrum $\mathrm{fil}_{\textup{Ad}}^*\BP \langle n \rangle:=\lim_{\Delta}\tau_{\ge 2*}(\BP \langle n \rangle \otimes \bF_p^{\otimes \bullet+1})$ is an $\bE_1$ $\mathrm{fil}_{\textup{Ad}}^*\BP$-algebra by~\cite{PP} where $\mathrm{fil}_{\textup{Ad}}^*\BP =\lim_{\Delta}\tau_{\ge 2*}(\BP \otimes \bF_p^{\otimes \bullet+1})$. The $\EE_1$-page can be identified with $\bF_p\langle \sigma v_0,\sigma v_1,\cdots ,\sigma v_n\rangle[\mu]$ by~\cite[Lemma 4.1.3]{HW18} as an $\bE_0$ algebra over the $\EE_1$-page $\bF_p\langle \sigma v_0,\sigma v_1,\cdots \rangle$ of the Hochschild--May spectral sequence for $\BP$, see~\cite[Proposition~2.2.2]{AKHW24}. Here the $\bE_0$ algebra structure is given by the canonical quotient map by $\sigma v_j$ for $j\ge n+1$. The bidegrees of the algebra generators $\sigma v_i$ are $(2p^i-1,2p^i-2)$ and the bidegrees of the the module generators $\mu^k$ are $(2k,0)$ where we refer to the first coordinate as the degree and the second coordinate as the Hochschild--May filtration. The $d_r$-differentials decrease degree by one and increase the Hochschild--May filtration by $r$. Since we know the abutment in the case of $\mathrm{BP}$, there are differentials $d_{2p^i-2}(\mu^{p^i})=\sigma v_i$ for each $i\ge 0$ in the Hochschild--May spectral sequence for $\BP$, see~\cite[Proposition~2.2.2]{AKHW24}. The $\bE_0$ algebra structure implies the differentials $d_{2p^i-2}(\mu^{p^i})=\sigma v_i$ and $d_{2p^j-2}(\mu^{p^j})=0$ as well as the differentials generated by the module structure. After considering these differentials as well as bidegrees, we observe that all remaining classes are infinite cycles. Again, we observe that there is no room for hidden $\bE_0$ $\THH_*(\BP;\mathbb{F}_p)$-algebra extensions. The map of Hochschild--May spectral sequences induced by $\BP \langle n \rangle \to \bF_p$ is the canonical quotient by $\sigma v_i$ for $0\le i\le n$ on the $\EE_1$-page and we observe that it is the canonical quotient by $\lambda_i$, detected by $\sigma v_{i-1}\mu^{p^{i}-p^{i-1}}$ in the $\EE_\infty$-page of the spectral sequence, for $1\le i\le n+1$ followed by the canonical inclusion $\bF_p[\mu^{p^n}]\subset \bF_p[\mu]$ and therefore we name classes accordingly. 
\end{proof}

This extends~\cite[Proposition~2.9]{ACH21} to arbitrary $\mathbb{E}_1$-$\MU$-algebra forms of $\BP\langle n\rangle$. 
 
\subsection{The motivic filtration on Hochschild homology}\label{sec:motHH}
The purpose of this section is to compute the homotopy groups of $\grmot^* \THH(\BP \langle n \rangle;\bF_p)$, 
defined using the $\THH(\MU)$-module structure on $\THH(\BP \langle n \rangle;\bF_p)$.
These bigraded homotopy groups form the $\EE_2$-page of the \emph{motivic spectral sequence}
\[
\mathrm{E}_2^{s,2w-s}=\pi_s\grmot^w \THH(\BP \langle n \rangle;\bF_p)\implies \THH_{s}(\BP \langle n \rangle;\bF_p)
\]
converging to the results of the previous subsection by~\cite[Corollary~A.2.9]{HRW22}. Recall that $s$ is called the \emph{degree}, $w$ is called the \emph{weight} and $2w-s$ is called the Adams weight. With the choice of coordinates $(s,2w-s)$ given by (degree, Adams weight), the differential convention of the spectral sequence is the standard Adams convention:
\[ 
	d_r:\mathrm{E}_2^{s,2w-s} \to \mathrm{E}_2^{s-1,2w-s+r}\,.
\]

We will deduce that the motivic spectral sequence 
degenerates at the $\EE_2$-page without extensions. 

\begin{rec}
Recall from \cite[Proposition~2.2.6, Proposition~2.2.9]{AKHW24} that 
\[ \pi_*\grmot^*\THH(\MU)=\bF_p\langle \lambda_i' : i\ge 1\rangle \qquad \text{ and } \qquad \pi_*\grmot^*\THH(\BP)=\bF_p\langle \lambda_j : j\ge 1\rangle 
\]
where $\|\lambda_i'\|=(2i+1,1)$, $\|\lambda_j\|=(2p^j-1,1)$ and using \cite{BM13}, we observe that the $\mathbb{E}_4$ $\MU$-algebra structure on $\BP$ induces an $\mathbb{E}_3$ algebra map 
\[ \pi_*\grmot^*\THH(\MU)\to  \pi_*\grmot^*\THH(\BP)\]
sending $\lambda_i'$ to $\lambda_j$ when $i=p^j-1$ and $0$ otherwise. 
Moreover, by \cite[Proposition~2.2.8]{AKHW24}, 
\[\pi_*\grmot^*\THH(\bF_p)=\bF_p[\mu]
\]
where $\|\mu\|=(2,0)$. 
\end{rec}

We need a preliminary result, which uses the notion of \emph{even flatness} due to 
Pstr\k{a}gowski. 

\begin{defin}[{\cite[Definitions~2.2,~4.4]{Pst23}}]
Let $R$ be an $\mathbb{E}_1$-ring, we say that a left $R$-module is \emph{perfect even} if it belongs to the smallest sub-$\infty$-category of the $\infty$-category of $R$-modules containing $\Sigma^{2k}R$ for all integers $k$ that is closed under extensions of retracts. We say that a left $R$-module is \emph{even flat} if it can be written as a filtered colimit of perfect even $R$-modules. 
\end{defin}

\begin{proposition}\label{algebraic eff map}
The $\THH_*(\MU;\bF_p)$-module $\THH_*(\BP\langle n\rangle;\bF_p)$ is even flat. 
In particular, the spectrum $\THH_*(\BP\langle n\rangle/\MU^{\otimes q+1})_p^{\wedge}$ is a finitely generated free $\mathbb{Z}_p[v_1,\cdots ,v_n]$-module concentrated in even degrees.
\end{proposition}
\begin{proof}
It suffices to check that 
\[
\Tor_{*}^{\THH_{*}(\MU;\bF_p )}(\THH_{*}(\BP\langle n\rangle;\bF_p ),\bF_p )
\] 
is concentrated in even total degrees by~\cite[Proposition~4.20,Theorem~4.21]{Pst23}. 
This follows from \autoref{prop:Hocschild-MaySSBPn}. Consequently, we conclude by a similar argument that 
\[\THH_*(\BP\langle n\rangle /\MU^{\otimes q+1};\bF_p)\] 
is concentrated in even degrees 
and since  $|v_i|=2p^i-2$ the $v_i$-Bockstein spectral sequences for $i=0,1,\cdots, n$ are also concentrated in even degrees. Therefore $\THH(\BP\langle n\rangle /\MU^{\otimes q+1})_p^{\wedge}$ is even and its homotopy groups are a finitely generated free $\mathbb{Z}_p[v_1,\cdots ,v_n]$-module using the collapsing $p$, $v_{1}$-, $\cdots$ $v_{n-1}$- and $v_n$-Bockstein spectral sequences. 
\end{proof}

\begin{defin}\label{May--Ravenel}
Let $(A,\Sigma)$ be a connected graded flat Hopf algebroid $(A,\Sigma)$ and let $M^*$ be a graded $(A,\Sigma)$-comodule with cobar complex 
\[ 
M^*\otimes_{A}\Sigma^{\otimes_A \bullet +1} \,.
\] 
Suppose $(A,\Sigma)$ and $M^*$ are equipped with compatible filtrations as in \cite[Definition~A1.3.7]{Rav86} so that there is a filtered cobar complex 
\[ 
\mathrm{Fil}^*M^*\otimes_{\mathrm{Fil}^*A}{\mathrm{Fil}^*\Sigma}^{\otimes_{\mathrm{Fil}^*A} \bullet +1} 
\] 
with associated graded 
\[ \mathrm{Gr}^*M^*\otimes_{\mathrm{Gr}^*A}{\mathrm{Gr}^*\Sigma}^{\otimes_{\mathrm{Gr}^*A} \bullet +1} 
\] 
then there is a spectral sequence 
\[ 
\EE_{2}^{t-s,t,f} =\Cotor_{(A,\Sigma)}^{s,t}(A,\mathrm{Gr}^fM^*)\implies \Cotor_{(A,\Sigma)}^{s,t}(A,M^*)\]
where 
\[ x\in [\mathrm{Gr}^*M^t\otimes_{\mathrm{Gr}^*A}\mathrm{Gr}^*\Sigma^{\otimes_{\mathrm{Gr}^*A} s +1} ]^{f}\]
is in tri-degree $(t-s,s,f)$ where $t-s$ is the degree, $s$ is the Adams weight and $f$ is the May--Ravenel filtration. 
With this convention, the differential satisfies $|d_r(x)|=(t-s-1,s+1,f+r)$, in other words they decrease degree by one, increase Adams weight by one and increase May--Ravenel filtration by $r$. 
\end{defin}

\begin{notation}
Write $(\bF_p,\Gamma)$ for the flat connected Hopf algebroid $(\bF_p,\bF_p\otimes_{\mathrm{THH}(\MU;\bF_p)}^{\mathbb{L}}\bF_p)$. 
\end{notation}

\begin{remark}\label{cotor-computation}
Note that 
\[\Cotor_{(\bF_p,\Gamma)}^{s,2w}(\bF_p,\bF_p)=\pi_{n}\mathrm{gr}_{\textup{mot}}^w\THH(\MU;\bF_p)\] 
by Koszul duality where $n=2w-s$ and 
\[\pi_{*}\mathrm{gr}_{\textup{mot}}^*\THH(\MU;\bF_p)=\mathbb{F}_p\langle \lambda_1',\lambda_2',\cdots \rangle \,.\]
\end{remark}

\begin{prop}\label{prop:thh-bpn} 
We can identify 
	\[
 \pi_*\gr^{*}_{\mot}\THH(\BP\langle n\rangle ;\bF_p )
	=	\bF_p\langle \lambda_1, \cdots, \lambda_{n+1}\rangle [\mu^{p^{n+1}}]
	\]
as an $\bE_0$ 
$\pi_*\gr^{*}_{\mot}\THH(\MU;\bF_p )$-algebra where the bidegrees of the $\pi_*\gr^{*}_{\mot}\THH(\MU;\bF_p )$-module generators are 
	\begin{align*}
	\|\mu^{sp^{n+1}}\| &= (2p^{n+1}s, 0) 
	\end{align*}
	for $s\ge 0$ where we write $\mu^{sp^{n+1}}$ for a class that has image $\mu^{sp^{n+1}}$ in 
	$\pi_*\gr^{*}_{\mot}\THH(\bF_p)=\bF_p[\mu ] $. 
The $\bE_0$ algebra structure map $\pi_*\grmot^*\THH(\MU;\bF_p)\to \pi_*\grmot^*\THH(\BP\langle n\rangle;\bF_p )$ 
is given by 
\[
		\lambda_{i}^{\prime}\mapsto \begin{cases} \lambda_{j} & \text{ if } i=p^{j}-1, j=1,\cdots ,n+1 \\ 0 & \text{ otherwise} \end{cases}
\]
\end{prop}

\begin{proof}
By \autoref{algebraic eff map} and~\cite[Lemma~2.2.17]{AKHW24}, it suffices to observe that 
\[ 
\Cotor_{(\bF_p,\Gamma)}^{*,*}(\bF_p,\THH_{*}(\mathrm{BP}\langle n\rangle ;\bF_p ) \otimes_{\THH_*(\MU;\bF_p)}^{\mathbb{L}}\bF_p )=\bF_p\langle \lambda_1,\cdots ,\lambda_{n+1}\rangle[\mu^{p^{n+1}}]
\]
as an $\bE_0$ $\Cotor_{(\bF_p,\Gamma)}^{*,*}(\bF_p,\bF_p)$-algebra and this identification is compatible with the map 
\[
\begin{tikzcd}
\Cotor_{(A,\Gamma)}^{*,*}(A,\THH_*(\BP\langle n\rangle;\bF_p)\otimes_{\THH(\mathrm{MU};\bF_p)}^{\mathbb{L}}\bF_p)\arrow{d} \\  
\Cotor_{(A,\Gamma)}^{*,*}(A,\THH_*(\bF_p)\otimes_{\THH(\mathrm{MU};\bF_p)}^{\mathbb{L}}\bF_p)
\end{tikzcd}
\]
of $\bE_0$ $\Cotor_{(\bF_p,\Gamma)}^{*,*}(\bF_p,\bF_p)$-algebras where 
\[ \Cotor_{(A,\Gamma)}^{*,*}(A,\THH_*(\bF_p)\otimes_{\THH(\mathrm{MU};\bF_p)}^{\mathbb{L}}\bF_p)\cong \bF_p [\mu] \,. \qedhere \]
\end{proof}

\begin{prop}\label{prop:agreement-of-quotients-BPn}
There is an equivalence of $\gr^*_{\mot} \THH(\MU)$-modules
\[
\gr^*_{\mot} \THH(\BP\langle n\rangle ;\bF_p)  \simeq \left ( \gr^*_{\mot} \THH(\BP \langle n\rangle ) \right) / (p,v_1,\cdots,v_n)
\]
\end{prop}

\begin{proof}
By \autoref{algebraic eff map}, we know that $\THH_*(\BP \langle n\rangle/\mathrm{MU}^{\otimes q+1})_p^{\wedge}$ is a finitely generated free $\mathbb{Z}_p^{\wedge}[v_{1},v_{2},\cdots ,v_{n}]$-module concentrated in even degrees. 
Letting $B_i=\BP\langle i\rangle$ for $-1\le i\le n$, we therefore have cofiber sequences 
\[ 
\grmot^*(\Sigma^{2p^i-2}\THH(\BP\langle n\rangle;B_{i-1}))\to \grmot^*\THH(\BP\langle n\rangle;B_{i-1})\to \grmot^{*}\THH(\BP\langle n\rangle;B_i)
\]
induced by the short exact sequences
\[ 
0\to \pi_*\Sigma^{2p^i-2}\THH(\BP\langle n\rangle;B_{i-1})\to \pi_*\THH(\BP\langle n\rangle;B_{i-1})\to \pi_*\THH(\BP\langle n\rangle;B_{i}) \to 0 \,.
\]
We can therefore identify
\[
\grmot^{*}\THH(\BP\langle n\rangle;B_i)\simeq (\grmot^{*}\THH(\BP\langle n\rangle;B_{i-1}))/(v_i)
\]
for each $0\le i\le n$ and after a finite induction we have proven the claim. 
\end{proof}

Note that the commutative diagram of $\MU$-modules 
\[
\begin{tikzcd}
\MU \ar[r] \ar[d] & \THH(\MU) \ar[d] \\ 
\BP\langle n\rangle \ar[r] \ar[d]  & \THH(\BP\langle n\rangle) \ar[d] \\
\bF_p \ar[r] \ar[dr]  & \THH(\BP\langle n\rangle;\mathbb{F}_p) \ar[d] \\
 & \THH(\bF_p) 
\end{tikzcd}
\]
implies that the class $v_{n+1}\in \pi_*\grmot^*\MU$ acts trivially on $\grmot^*\THH(\BP\langle n\rangle;\mathbb{F}_p)$ compatibly with the map $\THH(\BP\langle n\rangle;\mathbb{F}_p)\to \THH(\bF_p)$. 

\begin{defin}\label{def-v_n}
Consider the short exact sequence 
\[ 0 \to M \to  M\langle \varepsilon_{n+1}\rangle \to \Sigma^{2p^{n+1}-1,-1}M \to 0 
\]
where $M=\gr^*_{\mot} \THH(\BP\langle n\rangle ;\bF_p)$ induced by the cofiber sequence 
\[ \Sigma^{2p^{n+1}-2}\gr_{\ev}^*\MU \overset{v_{n+1}}{\longrightarrow} \gr_{\ev}^*\MU \to \gr_{\ev}^*\MU/v_{n+1}  \]
by base-change and using the observation that $v_{n+1}$ acts trivially on $M$. 
Then the class $\varepsilon_{n+1}$ is defined as a lift of $1\in \Sigma^{2p^{n+1}-1,-1}M$. The choice of lift $\varepsilon_{n+1}$ is unique (up to a unit) by \autoref{prop:thh-bpn}. 
\end{defin}

\begin{thm}\label{thm:map-to-fp-from-bpn}
There is an isomorphism 
\[ 
\pi_*((\grmot^*\THH(\BP\langle n\rangle))/(p,v_1,\cdots,v_{n+1}))\cong \bF_p\langle \lambda_{1},\cdots ,\lambda_{n+1}\rangle[\mu^{p^{n+1}}]\langle \varepsilon_{n+1}\rangle
\]
of $\bF_p\langle \lambda_1,\lambda_2,\cdots ,\lambda_{n+1}\rangle$-modules and the canonical map 
\[ 
\pi_*((\grmot^*\THH(\BP\langle n\rangle))/(p,v_1,\cdots,v_{n+1}))\to \pi_*((\grmot^*\THH(\bF_p))/(p,v_1,\cdots ,v_{n+1}))
\]
of $S^1$-equivariant 
$\bE_0$ $\pi_*((\grmot^*\THH(\MU))/(p,v_1,\cdots ,v_n))$-algebras 
is injective modulo $(\lambda_1,\cdots ,\lambda_{n+1})$ with image exactly $\bF_p[\mu^{p^{n+1}}]\langle \varepsilon_{n+1}\rangle$. Here the bidegree of $\varepsilon_{n+1}$ is 
\[ 
\| \varepsilon_{n+1}\|=(2p^{n+1}-1,-1) \,.
\]
\end{thm}
\begin{proof}
The canonical map 
\[ 
\pi_*((\grmot^*\THH(\BP \langle n \rangle))/(p,v_1,\cdots,v_n))\to \pi_*((\grmot^*\THH(\bF_p))/(p,v_1,\cdots ,v_n))
\]
factors through the map
\[
\pi_*((\grmot^*\THH(\BP \langle n \rangle))/(p,v_1,\cdots,v_n))\to \pi_*\grmot^*\THH(\bF_p)
\]
with image $\bF_p[\mu^{p^{n+1}}]$ by \autoref{prop:agreement-of-quotients-BPn} and \autoref{prop:Hocschild-MaySSBPn}. This determines the image of the map 
\[ \pi_*((\grmot^*\THH(\BP\langle n\rangle))/(p,v_1,\cdots,v_n))\to \pi_*((\grmot^*\THH(\bF_p))/(p,v_1,\cdots ,v_n))
\]
as well. Taking a further quotient by $v_{n+1}$ produces the result by \autoref{def-v_n} and the discussion preceding it. 
\end{proof}

  \section{Hodge--Tate Cohomology}
First, we recall the definition of Hodge--Tate cohomology in our context. 
\begin{defin}\label{def:tCp-motivic}
Given an $\mathbb{E}_1$-$\mathrm{MU}$-algebra $R$ such that $\mathrm{THH}(R/\mathrm{MU}^{\otimes q+1})_p^{\wedge}$ is even and torsion free for each integer $q\ge 0$, we define 
\[\fil^*_{\mot} \THH(R)^{tC_p} =\lim_{\Delta }\tau_{\ge 2*} (\mathrm{THH}(R/\mathrm{MU}^{\otimes \bullet+1})^{tC_p}) \,.\]
We call the associated graded 
\[\gr^*_{\mot} \THH(R)^{tC_p}:= \fil^*_{\mot} \THH(R)^{tC_p}/\fil^{*+1}_{\mot} \THH(R)^{tC_p} \]
the \emph{Hodge--Tate cohomology} of $R$. 
\end{defin}
\begin{remark}
By \autoref{algebraic eff map}, we know that $R$ satisfies the hypotheses of~\autoref{def:tCp-motivic}. By~\cite[Example~,~Corollary~A.2.11]{HRW22}, the filtration from \autoref{def:tCp-motivic} is compatible with the motivic filtration
\[\fil^*_{\mot} \THH(R)\simeq \lim_{\Delta }\tau_{\ge 2*} \mathrm{THH}(R/\mathrm{MU}^{\otimes \bullet+1})\,.\]
The assumption that  $\mathrm{THH}(R/\mathrm{MU}^{\otimes q+1})_p^{\wedge}$ is even and torsion free is used to determine that the spectrum $\mathrm{THH}(R/\mathrm{MU}^{\otimes q+1})^{tC_p}$ is even for each integer $q\ge 0$. 
\end{remark}

The main theorem of this section is the following calculation of mod $(p,v_1,\cdots,v_{n+1})$ Hodge--Tate cohomology:

\begin{thm}\label{thm:Hodge-Tate}
The commutative diagram
\[
\begin{tikzcd}
\pi_*((\gr^*_{\mot} \THH(\BP\langle n\rangle)) / (p,v_1,\cdots,v_{n+1})) \arrow{r}{\varphi_p} \arrow{d} & \pi_*((\gr^*_{\mot} \THH(\BP \langle n\rangle)^{tC_p} )/ (p,v_1,\cdots,v_{n+1})) \arrow{d} \\ 
\pi_*((\gr^*_{\mot}\THH(\bF_p) )/ (p,v_1,\cdots,v_{n+1})) \arrow{r}{\varphi_p} & \pi_*((\gr^*_{\mot} \THH(\bF_p)^{tC_p} )/ (p,v_1,\cdots,v_{n+1}) )\,,
\end{tikzcd}
\]
where the vertical maps are induced by the reduction map $\BP\langle n\rangle \to \bF_p$ and the horizontal maps are the Frobenius maps, 
may be identified as $\bF_p\langle \lambda_1,\lambda_2,\cdots, \lambda_{n+1}\rangle$-modules with the square
\[
\begin{tikzcd}
\bF_p\langle\lambda_1,\lambda_2,\cdots,\lambda_{n+1}\rangle[\mu^{p^{n+1}}]\langle \varepsilon_{n+1}\rangle  \arrow{r} \arrow{d} & \bF_p\langle \lambda_1,\lambda_2,\cdots ,\lambda_{n+1}\rangle
[\mu^{\pm p^{n+1}}]\langle \varepsilon_{n+1}\rangle \arrow{d} \\
\bF_p[\mu]\langle \epsilon_0,\epsilon_1,\cdots,\epsilon_{n+1}\rangle \arrow{r} & \bF_p[\mu^{\pm 1}]\langle \epsilon_0,\epsilon_1,\cdots,\epsilon_{n+1}\rangle 
\end{tikzcd}
\]
where the horizontal maps are given by inverting $\mu^{p^{n+1}}$ and $\mu$ respectively and the vertical maps are the tensor product of the canonical quotient $\bF_p\langle\lambda_1,\lambda_2,\cdots,\lambda_{n+1}\rangle\to \bF_p$ with the canonical inclusion. 
\end{thm}

First, we prove the result before passing to the motivic filtration. Here we use the map 
\[ \THH(R ;M)  \to \THH(R ;M^{\otimes p})^{tC_p} \] 
from~\cite[Theorem~6.3.1]{KMN23}, which is functorial in pairs $(R,M)$ where $R$ is an $\mathbb{E}_1$-ring and $M$ is an $R$-bimodule and exact by~\cite[Theorem~1.7]{NS18}. 
\begin{proposition}\label{Hodge-Tate-nonmotivic}
The commutative diagram
\[
\begin{tikzcd}
\pi_*\THH(\BP\langle n\rangle ;\bF_p)  \arrow{r}{\varphi_p} \arrow{d} & \pi_*\gr^*_{\mot} \THH(\BP \langle n\rangle;\bF_p^{\otimes p})^{tC_p} \arrow{d} \\ 
\pi_*\THH(\bF_p) \arrow{r}{\varphi_p} & \pi_*\gr^*_{\mot} \THH(\bF_p)^{tC_p} \,,
\end{tikzcd}
\]
where the vertical maps are induced by the reduction map $\BP\langle n\rangle \to \bF_p$ and the horizontal maps are the Frobenius maps, 
may be identified as $\bF_p\langle \lambda_1,\lambda_2,\cdots, \lambda_{n+1}\rangle$-modules with the square
\[
\begin{tikzcd}
\bF_p\langle\lambda_1,\lambda_2,\cdots,\lambda_{n+1}\rangle[\mu^{p^{n+1}}] \arrow{r} \arrow{d} & \bF_p\langle \lambda_1,\lambda_2,\cdots ,\lambda_{n+1}\rangle
[\mu^{\pm p^{n+1}}]\arrow{d} \\
\bF_p[\mu] \arrow{r} & \bF_p[\mu^{\pm 1}]
\end{tikzcd}
\]
where the horizontal maps are given by inverting $\mu^{p^{n+1}}$ and $\mu$ respectively and the vertical maps are the tensor product of the canonical quotient $\bF_p\langle\lambda_1,\lambda_2,\cdots,\lambda_{n+1}\rangle\to \bF_p$ with the canonical inclusion. 
\end{proposition}
\begin{proof}
First, note that by construction there is a map
\[ L_{n}^f\BP\longrightarrow  L_{n}^f\BP\langle n\rangle\]
and since $\tau_{\ge 0}L_{n}^f\BP\langle n\rangle\simeq \BP\langle n\rangle$ there is a canonical map 
\[ \tau_{\ge 0}L_{n}^f\BP \longrightarrow  \tau_{\ge 0}L_{n}^f\BP\langle n\rangle\simeq \BP\langle n\rangle \,.
\]
By \cite[Theorem~1.9]{Hov95} and~\cite[Theorem~1.5.4]{Hov97}, we know $v_{n+1}\in \pi_{2p^{n+1}-2}\MU$ maps trivially in $\pi_{2p^{n+1}-2}L_{n}^{f}\BP$, we can consider the element 
\[
\sigma^{2}v_{n+1}\in \pi_{2p^{n+1}}\THH(\tau_{\ge 0}L_n^f\BP/\MU) \,,
\]
which maps to a non-zero class 
\[
\sigma^{2}v_{n+1}=\mu^{p^{n+1}}\in \pi_{2p^{n+1}}\THH(\bF_p/\MU)\,. 
\]
This produces an element  $\mu^{p^{n+1}}$ in the bigraded commutative ring 
\[
\pi_*\grmot^*\THH(\tau_{\ge 0}L_n^{f}\BP)
\] 
which acts on $\pi_*\grmot^*\THH(\BP\langle n\rangle)$. It therefore makes sense to consider the map 
\[ \pi_*\grmot^*\THH(\BP\langle n\rangle ;\bF_p) \longrightarrow \pi_*\grmot^*\THH(\BP\langle n\rangle ;\bF_p^{\otimes p})^{tC_p}
\]
as a map of $\bF_p[\mu^{p^{n+1}}]$-modules and it makes sense to ask whether the map is given by inverting the class $\mu^{p^{n+1}}$. 

As in the proof of \autoref{prop:Hocschild-MaySSBPn}, given a spectrum $X$, write $\fil_{\textup{Ad}}^*X$ for the filtered spectrum
\[ 
\lim_{\Delta }\tau_{\ge *} (X \otimes \bF_p^{\otimes \bullet+1}) \,.
\]
and write $\gr_{\textup{Ad}}^*X$ for its associated graded. Recall that by~\cite{PP}, when $R$ is an $\mathbb{E}_m$-ring and $S$ is an $\mathbb{E}_{m-1}$-$R$-algebra then $\fil_{\textup{Ad}}^*R$ is an $\mathbb{E}_{m-1}$ $\fil_{\textup{Ad}}^*S$-algebra and similarly $\gr_{\textup{Ad}}^*R$ is an $\mathbb{E}_{m-1}$ $\gr_{\textup{Ad}}^*S$-algebra for $m\ge 1$. As in \cite[Recollection~4.3.2]{HW22}, we can determine that $\gr_{\textup{Ad}}^*\BP$ can be identified with $\bF_p\otimes \bS[v_0,v_1,\cdots]$ as an $\mathbb{E}_2$-algebra and $\gr_{\textup{Ad}}^*\BP\langle n\rangle$ can be identified with $\bF_p\otimes \bS[v_0,v_1,\cdots ,v_n]$ as $\bE_1$ $\gr_{\textup{Ad}}^*\BP$-algebra where $|v_i|=2p^i-2$ for $i\ge 1$. 

By \cite[Proposition~4.1.4]{HW22}, the map 
\[ 
L_p\THH(\bS[x];\bS)\to \THH(\bS[x];\bS^{\otimes p})^{tC_p} 
\]
is the $p$-completion map when $|x|=2j$ for some integer $j>0$; here, one defines $L_p$ as in~\cite[Appendix~A]{AMMN22}. By~\cite[Lemma~4.1.3]{HW22}, we can also identify
\[ \bF_p\otimes \THH(\bS[x];\bS)=\bF_p\langle \sigma v_i \rangle \]
when $|x|=2j$ for some integer $j>0$. We can therefore consider the $\mathbb{E}_0$-$\gr_{\textup{Ad}}^*\mathrm{BP}$-algebra structure to show that produce an action of $\mu$ on the map of $\EE_2$-pages 
\begin{equation}\label{map of spectral sequences}
\begin{tikzcd}
\pi_*\THH(\bF_p)\otimes_{\bF_p}\pi_*\bF_p\otimes L_p\THH(\bS[v_0,\dots ,v_n];\bS)\arrow{d} \\
\THH_*(\bF_p)^{tC_p}\otimes_{\bF_p}\pi_*\bF_p\otimes \THH(\bS[v_0,\dots ,v_n];\bS^{\otimes p})^{tC_p} \,.
\end{tikzcd}
\end{equation}
associated to the map of filtered spectra 
\[
L_p\THH(\mathrm{fil}_{\textup{Ad}}^*\mathrm{BP}\langle n\rangle; \mathrm{fil}_{\textup{Ad}}^*\bF_p)\longrightarrow 
\THH(\mathrm{fil}_{\textup{Ad}}^*\mathrm{BP}\langle n\rangle ;(\mathrm{fil}_{\textup{Ad}}^*\bF_p)^{\otimes p})^{tC_p}  \,,
\]
see~\cite[Appendix~A]{AMMN22} for the definition of $L_p$. Using this, the remarks above,  and the known computation 
\[ \THH_*(\bF_p)\to \THH_*(\bF_p)^{tC_p} \]
is given by the map $\bF_p[\mu]\to \bF_p[\mu,\mu^{-1}]$ inverting $\mu$, see~\cite[Proposition~4.3]{HM97}, we can show that this map is given by inverting $\mu$ on the $\EE_2$-pages, cf.~proof of \cite[Proposition~4.2.1]{HW22}. 
We have completely computed the top spectral sequence of \eqref{map of spectral sequences} in \autoref{prop:Hocschild-MaySSBPn} and since the bottom spectral sequence is given by inverting $\mu$ on $\EE_2$-pages, and all classes except powers of $\mu$ are infinite cycles, 
we determine that at each page the spectral sequence is given by inverting some power of $\mu$, and hence remains a monomorphism entirely determining also the spectral sequence on the bottom. 
This proves the claim. 
\end{proof}

\begin{proof}[Proof of \autoref{thm:Hodge-Tate}]
It suffices to show that the map 
\begin{equation}\label{Frobenius grmot} \pi_*((\grmot^*\mathrm{THH}(\mathrm{BP}\langle n\rangle ))/(p,\cdots ,v_n)) \to \pi_*((\grmot^*\mathrm{THH}(\mathrm{BP}\langle n\rangle )^{tC_p})/(p,\cdots ,v_n)) \end{equation}
of motivic spectral sequences is a map of collapsing motivic spectral sequences. The claim then follows by quotienting by $v_{n+1}$ on the commutative diagram
\[
\begin{tikzcd}
\pi_*((\grmot^*\mathrm{THH}(\mathrm{BP}\langle n\rangle ))/(p,\cdots ,v_n))  \arrow{r}  \arrow[d,"="]  &  \pi_*((\grmot^*\mathrm{THH}(\mathrm{BP}\langle n\rangle )^{tC_p})/(p,\cdots ,v_n)) \arrow[d,"="] \\ 
\pi_*\grmot^*\mathrm{THH}(\mathrm{BP}\langle n\rangle ;\bF_p) \arrow{r}  \arrow{d}  &  \pi_*\grmot^*\mathrm{THH}(\mathrm{BP}\langle n\rangle ;\bF_p^{\otimes p} )^{tC_p} \arrow{d} \\ 
\pi_*\grmot^*\mathrm{THH}(\bF_p ) \arrow{r} \arrow{d} &  \pi_*\grmot^*\mathrm{THH}(\bF_p )^{tC_p} \arrow{d}\\ 
\pi_*((\grmot^*\mathrm{THH}(\bF_p ))/(p,\cdots ,v_n)) \arrow{r} & \pi_*((\grmot^*\mathrm{THH}(\bF_p )^{tC_p})(p,\cdots ,v_n) )
\end{tikzcd}
\]
of $\grmot^*\mathrm{THH}(\mathrm{MU})$-modules on which $v_{n+1}$-acts trivially. 

Consider the map of cobar complexes 
\begin{equation}\label{cobar complex computing motivic filtration}
\begin{tikzcd}
\pi_* \left ( \mathrm{THH}(\mathrm{BP}\langle n\rangle ;\bF_p)\otimes_{\mathrm{THH}(\mathrm{MU};\bF_p)}\bF_p^{\otimes_{\mathrm{THH}(\mathrm{MU};\bF_p)} \bullet +1}  \right ) \arrow{d} \\ 
\pi_* \left ( \mathrm{THH}(\mathrm{BP}\langle n\rangle ;\bF_p^{\otimes p})^{tC_p}\otimes_{\mathrm{THH}(\mathrm{MU};\bF_p^{\otimes p})^{tC_p}}((\bF_p^{\otimes p})^{tC_p})^{\otimes_{\mathrm{THH}(\mathrm{MU};\bF_p^{\otimes p})^{tC_p}} \bullet +1}  \right ) 
\end{tikzcd}
\end{equation}
whose cohomology is the map \eqref{Frobenius grmot}. 
By~\cite{LNR11} and~\cite{LNR12}, we know that the maps $\mathrm{THH}(\mathrm{MU};\bF_p)\to \mathrm{THH}(\mathrm{MU};\bF_p^{\otimes p})^{tC_p}$ and $\bF_p\to (\bF_p^{\otimes p})^{tC_p}$ are equivalences,  c.f.~\cite{HW21,AKZ25}.  Filtering by 
\begin{equation}\label{Hodge--Tate--May--Ravenel}
\begin{tikzcd}
\tau_{\ge \bullet} \mathrm{THH}(\mathrm{BP}\langle n\rangle ;\bF_p)\otimes_{\tau_{\ge \bullet} \mathrm{THH}(\mathrm{MU};\bF_p)}\tau_{\ge \bullet} \bF_p^{\otimes_{\tau_{\ge \bullet} \mathrm{THH}(\mathrm{MU};\bF_p)} \bullet +1}  \arrow{d} \\ 
\tau_{\ge \bullet} (\mathrm{THH}(\mathrm{BP}\langle n\rangle ;\bF_p^{\otimes p})^{tC_p})\otimes_{\tau_{\ge \bullet} (\mathrm{THH}(\mathrm{MU};\bF_p^{\otimes p})^{tC_p})}((\tau_{\ge \bullet} (\bF_p^{\otimes p})^{tC_p}))^{\otimes_{\tau_{\ge \bullet} (\mathrm{THH}(\mathrm{MU};\bF_p^{\otimes p})^{tC_p})} \bullet +1}  \,, 
\end{tikzcd}
\end{equation}
and doing the same with $\mathrm{BP}\langle n\rangle$ replaced by $\tau_{\ge 0}L_n^f\mathrm{BP}$ we determine that on associated graded the map is given by inverting $\mu^{p^{n+1}}$ by \autoref{Hodge-Tate-nonmotivic}. We then note that 
\[ \pi_*\mathrm{THH}(\mathrm{BP}\langle n\rangle ;\bF_p^{\otimes p})^{tC_p}=\pi_*\mathrm{THH}(\mathrm{BP}\langle n\rangle ;\bF_p)[\mu^{-p^{n+1}}] \]
is even flat over $\pi_*\mathrm{THH}(\mathrm{MU};\bF_p^{\otimes p})^{tC_p}\cong \pi_*\mathrm{THH}(\mathrm{MU};\bF_p)$ since it is a filtered colimit of perfect even $\pi_*\mathrm{THH}(\mathrm{MU};\bF_p)$-modules using~\autoref{algebraic eff map}. By a similar argument to~\autoref{algebraic eff map}, see \cite[Lemma~2.2.17]{AKHW24}, we also know that the spectral sequence associated to the filtration \eqref{Hodge--Tate--May--Ravenel} collapses at the $\EE_2$-page since the $\EE_2$-page is concentrated in even degrees.  We can therefore regard the $\EE_2=\EE_{\infty}$-page as the associated graded of a filtration on the cobar complex \eqref{cobar complex computing motivic filtration} and consider the associated May--Ravenel spectral sequence~\cite[Theorem~A1.3.9]{Rav86}.  
Since it is given by inverting a class $\mu^{p^{n+1}}$ in degree $2p^{n+1}$, Adams weight zero and May--Ravenel filtration $2p^{n+1}$, we can determine that the May--Ravenel spectral sequence converges using \autoref{cotor-computation}, the fact that 
$\pi_*\mathrm{gr}_{\textup{mot}}^*\mathrm{THH}(\mathrm{BP}\langle n\rangle;\bF_p)$ is connective and concentrated in finitely many Adams weights and applying \cite[Lemma~2.34]{CM21} (cf.~\cite[Theorem~2.13]{MRS01}). For the relevant grading conventions of this spectral sequence, see \autoref{May--Ravenel}. We determine that the $\EE_2$-page of the May--Ravenel spectral sequence can be identified with the abutment and therefore we determine that the map of motivic spectral sequences is a map of collapsing motivic spectral sequences and it is given by inverting $\mu^{p^{n+1}}$ as desired. 
\end{proof}

 \section{Prismatic cohomology}\label{sec:prismatic}

Recall that there is a map  
\[ (\grmot^*\THH(\BP\langle n\rangle ))/(p,v_1,\cdots ,v_{n+1})\to(\grmot^*\THH(\bF_p))/(p,v_1,\cdots,v_{n+1})\]
of $\bE_0$ $(\grmot^*\THH(\MU))/(p,\cdots ,v_{n+1})$-algebras induced by the canonical $\bE_1$ $\BP$-algebra reduction map 
\[\BP\langle n\rangle\to \tau_{\le 0}\BP\langle n\rangle=\bZ_{(p)}\to \bF_p.\] 
We computed that this map is injective mod $(\lambda_1,\lambda_2,\cdots,\lambda_{n+1})$, with image $\bF_p[\mu^{p^n}]\otimes \Lambda (\varepsilon_{n+1})$, 
in \autoref{thm:map-to-fp-from-bpn}. This is the key fact we need in this section. 

First, we recall the definition of the associated graded of the motivic filtration used to define Nygaard completed prismatic cohomology and related invariants. 

\begin{defin}
Let $R$ be an $\mathbb{E}_1$-$\mathrm{MU}$-algebra such that $\mathrm{THH}(R/\mathrm{MU}^{\otimes q+1})_p^{\wedge}$ is even, then we define 
\[ \mathrm{fil}_{\textup{mot}}^*\mathrm{TP}(R)_p^{\wedge}=\lim_{\Delta} \tau_{\ge 2*} \mathrm{TP}(R/\mathrm{MU}^{\otimes \bullet+1})_p^{\wedge} \]
with associated graded $\mathrm{gr}_{\textup{mot}}^*\mathrm{TP}(R)_p^{\wedge}$. 
We call $\pi_*\mathrm{gr}_{\textup{mot}}^*\mathrm{TP}(R)_p^{\wedge}$
the \emph{Nygaard completed prismatic cohomology} of $R$. This is justified by \cite[\S~4, Corollary~A.2.11]{HRW22}. We also define 
\[ \filmot^*\TC^{-}(R)_p^{\wedge}= \lim_{\Delta} \tau_{\ge 2*}\TC^{-}(R/\mathrm{MU}^{\otimes q+1})_p^{\wedge}\]
with associated graded $\grmot^*\TC^{-}(R)_p^{\wedge}$. 
By~\cite[Example~4.2.3, Corollary~A.2.10]{HRW22}, there are filtered maps 
\[ \mathrm{can},\varphi :  \mathrm{fil}_{\textup{mot}}^*\TC^{-}(R)_p^{\wedge}\to  \mathrm{fil}_{\textup{mot}}^*\TP(R)_p^{\wedge}\]
converging to the canonical and Frobenius maps of~\cite{NS18} whose equalizer 
\[ \mathrm{eq} \left ( \xymatrix{ \mathrm{fil}_{\textup{mot}}^*\TC^{-}(R)_p^{\wedge}\ar@<1ex>[r]^{\textup{can}} \ar@<-1ex>[r]_{\varphi} &  \mathrm{fil}_{\textup{mot}}^*\mathrm{TP}(R)_p^{\wedge} }\right ) \]
can be identified with $\mathrm{fil}_{\textup{mot}}^*\TC(R)_p^{\wedge}$ with associated graded 
$\mathrm{gr}_{\textup{mot}}^*\TC(R)_p^{\wedge}$. 
We call 
\[\pi_*\mathrm{gr}_{\textup{mot}}^*\TC(R)_p^{\wedge}\] 
the ($\MU$-based) \emph{syntomic cohomology} of $R$ in light of~\cite[\S~4, Corollary~A.2.11]{HRW22}.
We will write 
\[  \mathrm{gr}_{\textup{mot}}^*F(R)/(p,\cdots ,v_m) = (\mathrm{gr}_{\textup{mot}}^*F(R)_p^{\wedge})/(p,\cdots ,v_m)\]
for $F\in \{\mathrm{THH},\mathrm{THH}^{tC_p}, \mathrm{TC}^{-}, \mathrm{TP}, \mathrm{TC}\}$ and $m\ge 0$ an integer since these agree by~\cite[Example~4.2.3, Corollary~A.2.6]{HRW22} and the same argument as \cite[Propositions~2.4.1]{HRW22}. Here, we use the fact that  $\mathrm{gr}_{\textup{mot}}^*F(R)_p^{\wedge}$ is a $\gr_{\ev}^*\mathbb{S}$-module for  $F\in \{\mathrm{THH},\mathrm{THH}^{tC_p}, \mathrm{TC}^{-}, \mathrm{TP}, \mathrm{TC}\}$. 
$\grmot^*\TC^{-}(R)_p^{\wedge}$ and  $\mathrm{gr}_{\textup{mot}}^*\mathrm{TP}(R)_p^{\wedge}$ are $\mathrm{gr}_{\textup{ev}}^*\mathbb{S}$-modules. We will refer to 
\[\pi_*\mathrm{gr}_{\textup{mot}}^*\mathrm{TC}(R)/(p,v_1,\cdots ,v_m)\] 
as the ($\MU$-based) \emph{mod $(p,v_1,\cdots,v_m)$ syntomic cohomology} of $R$. 
\end{defin}

\begin{defin}\label{def:Nygaard-filt}
Let $R$ be an $\mathbb{E}_1$-$\mathrm{MU}$-algebra such that $\mathrm{THH}(R/\mathrm{MU}^{\otimes q+1})_p^{\wedge}$ is even and $p$-torsion free. Let $m\ge 0$ be an integer.
By~\cite[Variant~A.1.8]{HRW22} (cf.~\cite[\S~6]{AKAR}) and  \cite[Lemma~II.4.2]{NS18}, there are compatible algebraic $t$-Bockstein spectral sequences  \footnote{We use the terminology from \cite[\S~6.5]{HRW22}.} 
\[
\begin{tikzcd}
\EE_2^{s,2w-s,*}=\pi_s\grmot^w\THH(R)/(p,v_1,\cdots ,v_{m})[t,t^{-1}]\arrow[r,Rightarrow] & \pi_*\grmot^*\TP(R)/(p,v_1,\cdots,v_{m})  \\
\EE_2^{s,2w-s,*}=\pi_s\grmot^w\THH(R)/(p,v_1,\cdots ,v_{m})[t] \arrow{u}\arrow{d}\arrow[r,Rightarrow] & \pi_*\grmot^*\TC^{-}(R)/(p,v_1,\cdots,v_{m})  \arrow{u}[swap]{\textup{can}} \arrow{d}{\varphi}\\
\EE_2^{s,2w-s,*}=\pi_s\grmot^w\THH(R)^{tC_p}/(p,v_1,\cdots ,v_{m})[t] \arrow[r,Rightarrow] &\pi_*\grmot^*\TP(R)/(p,v_1,\cdots,v_{m})  
\end{tikzcd}
\]
where  $t$ is in degree $-2$ and Adams weight zero. We refer to the filtration on these spectral sequences as the \emph{Nygaard filtration} and here $t$ is in Nygaard filtration $1$ and classes 
\[x\in \pi_*\grmot^*\THH(R)/(p,v_1,\cdots ,v_m)\] have Nygaard filtration zero. 
The differential convention is then 
\[ d_r :\EE_2^{s,2w-s,f}\to \EE_2^{s-1,2w-s+1,f+r} \]
where $s$ is the degree, $2w-s$ is the Adams weight and $f$ is the Nygaard filtration. 
\end{defin}

\begin{proposition}\label{thm:Hodge-Tateresults}
The algebraic $t$-Bockstein spectral sequence 
\[
\begin{tikzcd}
\EE_2^{s,2w-s,*}=\pi_s((\grmot^w\THH\BP\langle n\rangle)^{tC_p})/(p,v_1,\cdots ,v_{n+1}))[t]  \arrow[d,Rightarrow] \\
\pi_*((\grmot^*\TP(\BP\langle n\rangle))/(p,v_1,\cdots,v_{n+1}))
\end{tikzcd}
\]
has $\EE_2$-page 
\[
\pi_*((\grmot^*\THH(\BP\langle n\rangle)^{tC_p})/(p,v_1,\cdots ,v_{n+1})[t]=\bF_p\langle \lambda_1,\cdots ,\lambda_{n+1}\rangle [\mu^{\pm p^n},t]\langle \varepsilon_{n+1} \rangle 
\]
and it collapses after the differential
\[  
    d_1(\varepsilon_{n+1})=t\mu^{p^{n+1}}
\]
along with those differentials given by the Leibniz rule 
\[ 
    d_1(\varepsilon_{n+1}x)=t\mu^{p^{n+1}}x
\]
where $x\in \bF_p\langle \lambda_1,\cdots ,\lambda_{n+1}\rangle [\mu^{\pm p^n},t]$. 
We can therefore identify 
\[ 
 \pi_*((\grmot^*(\THH(\BP\langle n\rangle)^{tC_p})^{hS^1})/(p,v_1,\cdots,v_{n+1}))= \bF_p\langle \lambda_1,\cdots ,\lambda_{n+1}\rangle[\mu^{\pm p^{n+1}}]
\]
where we abuse notation and write the same name for the class in the abutment and the class in the $\EE_2$-page since each class is detected by the class of the same name without indeterminacy. 
\end{proposition} 

\begin{proof}
Since $v_{n+1}$ is detected by $t\mu^{p^{n+1}}$ as in~\cite[\S~5]{HW22}, we determine that $d_1(\varepsilon_{n+1})=t\mu^{p^{n+1}}$ by \autoref{def-v_n}, since $\varepsilon_{n+1}$ corresponds to a null homotopy of $v_{n+1}$.  
Using the action of $\lambda_{1},\lambda_{2}, \cdots ,\lambda_{n+1}$ together with the monomorphism into the spectral sequence for $\bF_p$ modulo $(\lambda_{1},\lambda_{2}, \cdots ,\lambda_{n+1})$, 
we determine a Leibniz rule 
$d_1(x\varepsilon_{n+1})=xt\mu^{p^{n+1}}$ for all 
\[
x\in \pi_*((\grmot^*\THH(\BP\langle n\rangle)^{tC_p})/(p,v_1,\cdots ,v_n))\,.
\] 
Since $\mu$ is a unit in an algebra spectral sequence  that acts on this spectral sequence, namely the corresponding spectral sequence for $\tau_{\ge 0}L_n^f\BP$ (cf. proof of \autoref{thm:Hodge-Tate}), 
the spectral sequence collapses after running the first differential and all permanent cycles are concentrated in Nygaard filtration zero. Moreover, we can conclude that 
\[ 
\pi_*((\grmot^*\TP(\BP\langle n\rangle))/(p,v_1,\cdots ,v_{n+1}))\cong \ker (\sigma )
\]
where 
\[
\begin{tikzcd}
\pi_{*}((\grmot^*\THH(\BP\langle n\rangle )^{tC_p})/(p,v_1,\cdots ,v_{n+1})) \arrow{d}{\sigma} \\
\pi_{*+1}((\grmot^{*+1}\THH(\BP\langle n\rangle )^{tC_p} )/(p,v_1,\cdots ,v_{n+1}))
\end{tikzcd}
\]
is the $\sigma$-operator defined as in~\cite[\S~2.1.2]{AKHW24} using the $\filev^*\mathbb{S}[S^1]$-action. 
\end{proof}

\begin{notation}
We write 
\[ \lambda_{W}=\prod_{w\in W}\lambda_w
\]
for $W\subset \{1,\cdots ,n+1\}$ where $\lambda_{W}=1$ if $W=\emptyset$. 
\end{notation}

\begin{prop}\label{prop:prismatic-bpn}
The algebraic $t$-Bockstein spectral sequence
\[
\begin{tikzcd}
\mathrm{E}_2^{s,2w-s,*}=\pi_s((\grmot^w\THH(R))/(p,v_1,\cdots ,v_{n+1}))[t,t^{-1}]\arrow[d,Rightarrow]  \\
\pi_*((\grmot^*\TP(R))/(p,v_1,\cdots,v_{n+1})) 
\end{tikzcd}
 \]
has differentials  
\[ 
d_{1}(x\varepsilon_{n+1})=xt\mu^{p^{n+1}}\,,
\]
for all $x\in \pi_*\gr_{\mot}^*\THH(\BP\langle n\rangle )/(p,v_1,\cdots ,v_{n})[t^{\pm 1}]$
and 
\[d_{p^m}(t^{p^{m-1}})=t^{p^{m}+p^{m-1}}\lambda_m\]
for all $1\le m \le n+1$ as well as those differentials generated by the Leibniz rule
\[
d_{p^m}(t^{jp^{m-1}}\lambda_S)\dot{=}t^{p^{m}+jp^{m-1}}\lambda_m\lambda_S
\]
for $0<j<p$ and $S \subset \{1,\cdots ,n+1\} -\{m\}$. 
The spectral sequence then collapses at the $p^{n+1}+1$-page without room for further differentials. Consequently, 
\[ 
\pi_*((\gr_{\mot}^*\TP(\BP\langle n\rangle ))/(p,v_1,\cdots ,v_{n+1}))\cong \mathbb{F}_p\langle \lambda_1,\lambda_2,\cdots ,\lambda_{n+1}\rangle [t^{\pm p^{n+1}}] 
\]
as $\mathbb{F}_p\langle \lambda_1,\lambda_2,\cdots ,\lambda_{n+1}\rangle$-modules. 
\end{prop}
\begin{proof}
The differential $d_{1}(\varepsilon_{n+1})=t\mu^{p^{n+1}}$ follows from \autoref{def-v_n}, i.e. $\varepsilon_{n+1}$ corresponds to a null homotopy of $v_{n+1}$, and \autoref{thm:map-to-fp-from-bpn}. The Leibniz rule 
\[d_{1}(x\varepsilon_{n+1})=xt\mu^{p^{n+1}}
\]
for all $x\in \pi_*((\gr_{\mot}^*\THH(\BP\langle n\rangle ))/(p,v_1,\cdots ,v_{n}))[t^{\pm 1}]$
follows from~\cite[Corollary 3.2.6]{AKHW24}, which implies that the classes $\lambda_1,\lambda_2,\cdots ,\lambda_{n+1}$ are permanent cycles that act on this spectral sequence, together with the monomorphism into the spectral sequence for $\bF_p$ modulo $(\lambda_1,\lambda_2,\cdots ,\lambda_{n+1})$. 

The differentials 
\[
d_{p^m}(t^{p^{m-1}}\lambda_S)=t^{p^{m}+p^{m-1}}\lambda_m\lambda_S,
\]
for all $S\subset \{1,\cdots ,n+1\}$ with $m\not\in S$ for each $m\ge 1$ follow from the action of the spectral sequence for $\BP$ and the differentials in that spectral sequence, computed in~\cite[Proposition~3.2.9]{AKHW24}. By~\cite[Proposition~3.2.9]{AKHW24}, we also  know that that $t^{p^{n+1}}$ survives to the $\EE_{p^{n+2}}$-page and it acts on the spectral sequence up until this page, so $\EE_{p^{n+1}+1}=\EE_{p^{n+2}}$.  
Finally, we simply check bidegrees and observe that the spectral sequence must collapse at the $\EE_{p^{n+2}}$-page. 
Note that we abuse notation and write $\lambda_1^{a_1}\cdot \lambda_2^{a_2}\cdot \ldots \cdot \lambda_{n+1}^{a_{n+1}}t^j$ for a choice of class detected by the class with the same name modulo higher Nygaard filtration. 
\end{proof}

\begin{remark}
Note that, as a consequence of \autoref{thm:Hodge-Tateresults}, we have an isomorphism 
\[ 
\pi_*((\grmot^*\TP(\BP\langle n\rangle))/(p,v_1,\cdots ,v_{n+1}))\cong \pi_*((\grmot^*\THH(\BP\langle n\rangle)^{tC_p})/(p,v_1,\cdots ,v_n))\,.
\]
This could alternatively be proven directly along the lines of~\cite[\S~6.4]{HRW22} and we believe that such a result holds quite generally, however one still needs to understand the Nygaard filtration on prismatic cohomology to compute syntomic cohomology. This is achieved in \autoref{prop:prismatic-bpn}
\end{remark}

  \section{Syntomic cohomology}\label{sec:syntomic}
We now have all the ingredients to compute the syntomic cohomology of $\BP\langle n\rangle$. First, we need to determine 
\[
\pi_*((\grmot^*\TC^{-}(\BP\langle n\rangle))/(p,v_1,\cdots,v_{n+1})) \,.
\] 

We note that these bigraded groups are equipped with the Nygaard filtration coming from the algebraic $t$-Bockstein spectral sequence of \autoref{def:Nygaard-filt} and therefore there is an exact sequence
\[
0 \to \mathrm{Nyg}_{\ge 1} \to \pi_*((\grmot^*\TC^{-}(\BP\langle n\rangle))/(p,v_1,\cdots,v_{n+1})) \to \mathrm{Nyg}_{=0} \to 0 
\]
where $\mathrm{Nyg}_{\ge 1}$ denotes the subgroup generated by classes in positive Nygaard filtration and $\mathrm{Nyg}_{=0}$ is defined to be the cokernel. Let $T=\ker (\can)$ and $F=\pi_*\grmot^*\TC^{-}(\BP\langle n\rangle)/\ker (\can)$. 

\begin{notation}
There is a commutative diagram 
\[
\begin{tikzcd}
& 0 \ar[d] & 0 \ar[d] & 0 \ar[d] &  \\ 
0 \ar[r] & A_{11}\ar[d] \ar[r] & \mathrm{T} \ar[d] \ar[r] &  A_{01} \ar[d]  \ar[r] & 0  \\
0 \ar[r] & \mathrm{Nyg}_{\ge 1}  \ar[r] \ar[d]& \pi_*\mathrm{TC}^{-}(\BP\langle n\rangle) \ar[r] \ar[d] & \mathrm{Nyg}_{=0} \ar[r] \ar[d] & 0 \\ 
0 \ar[r] & A_{10} \ar[d] \ar[r] & F \ar[r] \ar[d]& A_{00} \ar[r] \ar[d]& 0 \\ 
& 0  & 0  & 0   &  
\end{tikzcd}
\]
where
\begin{align*}
A_{11} &: =\mathrm{Nyg}_{\ge 1}\cap T \,,  \qquad A_{10}&: =\mathrm{Nyg}_{\ge 1}/A_{11}  \,,\qquad 
A_{01}&: =T/A_{11}  \qquad  \text{ and } \qquad
A_{00} &: = F/A_{10}  \,.
\end{align*} 
\end{notation}

\begin{proposition}\label{Nyg filtration crossing diff}
The classes in the $\EE_1$-page of the algebraic $t$-Bockstein spectral sequence that are hit by differentials that cross from negative Nygaard filtration to positive Nygaard filtration are exactly the classes contained in
\[
\bigoplus_{j=1}^{n+1}  \mathbb{F}_p\langle \lambda_{s} :1\le s\le n+1,s\ne j
\rangle \otimes \bF_p\{ t^{dp^{j-1}}\lambda_j : 0 < d < p \}\,.
\]
\end{proposition}
\begin{proof}
This follows from a careful bookkeeping of the differentials determined in \autoref{prop:prismatic-bpn}. 
\end{proof}

The following corollary is immediate from \autoref{prop:prismatic-bpn} and \autoref{Nyg filtration crossing diff}. 

\begin{corollary}\label{cor:TC-BPn}
There is a preferred isomorphism 
\begin{align*} 
\pi_*((\grmot^*\TC^{-}(\BP\langle n\rangle))/(p,v_1,\cdots,v_{n+1}))\cong A_{00}\oplus A_{10}\oplus A_{01}\oplus A_{11} 
\end{align*}
as $\mathbb{F}_p\langle \lambda_1,\lambda_2,\cdots \lambda_{n+1}\rangle$-modules where 
\begin{align}
\label{a00} A_{00}  &=\bF_p\langle \lambda_1,\lambda_2,\cdots ,\lambda_{n+1}\rangle\,, \\
\label{a01} A_{01} & = \bF_p\langle \lambda_1,\lambda_2,\cdots ,\lambda_{n+1}\rangle [\mu^{p^{n+1}}]\{\mu^{p^{n+1}}\}  \,, \\ 
\label{a10} A_{10} & = \mathbb{F}_p\langle \lambda_1,\lambda_2,\cdots ,\lambda_{n+1}\rangle[t^{p^{n+1}}]\{t^{p^{n+1}}\} \text{ and } \\
\label{a11}  A_{11} & =\bigoplus_{j=1}^{n+1}  \mathbb{F}_p\langle \lambda_{s} :1\le s\le n+1,s\ne j \rangle \otimes \bF_p\{\Xi_{j,d} : 0 < d < p \} \,. 
\end{align}

We abuse notation and write $\lambda_1^{a_1}\cdot \lambda_2^{a_2}\cdot \ldots \cdot \lambda_{n+1}^{a_{n+1}}$ for the class detected by the class with the same name for $a_1,a_2,\cdots ,a_{n+1}\in \{0,1\}$ and $s\ge 0$. This is consistent with our choice to give the lifts of these classes to $\pi_*((\mathrm{gr}_{\mot}^*\mathrm{TC}(\mathrm{MU}))/p)$ the same name. 
We write $\lambda_1^{a_1}\cdot \ldots \lambda_{j-1}^{a_{j-1}}\cdot \lambda_{j+1}^{a_{j+1}}\cdot \ldots \cdot \lambda_{n+1}^{a_{n+1}}\Xi_{j,d}$ for a class detected by $\lambda_1^{a_1}\cdot \ldots \lambda_{j-1}^{a_{j-1}}\cdot \lambda_{j+1}^{a_{j+1}}\cdot \ldots \cdot \lambda_{n+1}^{a_{n+1}}t^{dp^{j-1}}\lambda_{j}$  where $a_s\in \{0,1\}$ for each $1\le s\le n+1$ where $s\ne j$. Within $\mathrm{ket}(\varphi-\mathrm{can})$ there is no indeterminacy in this choice. 
We abuse notation and write $\lambda_1^{a_1}\cdot \lambda_2^{a_2}\cdot \ldots \cdot \lambda_{n+1}^{a_{n+1}}\mu^{sp^{n+1}}$ for a class detected by the class with same name where $s\ge 1$ and $a_1,a_2,\cdots ,a_{n+1}\in \{0,1\}$. Again there is no potential indeterminacy when considering classes in $\mathrm{ker}(\varphi-\mathrm{can})$. 

Moreover, the canonical map 
\[ 
    \textup{can} : \pi_*((\grmot^*\TC^{-}(\BP\langle n\rangle ))/(p,v_1,\cdots ,v_{n+1})) \to \pi_*((\grmot^*\TP(\BP\langle n\rangle ))/(p,v_1,\cdots ,v_{n+1}))
\]
sends $\lambda_k$ to $\lambda_k$, $t^{p^{n+1}}$ to $t^{p^{n+1}}$, and sends $\lambda_S\cdot \Xi_{j,d}$ and $\lambda_T\mu^{kp^{n+1}}$ to zero for each $k\ge 1$, $S\subset \{1,2,\cdots, n+1\}-\{j\}$, $j\in \{1,2,\cdots ,n+1\}$, $0<d<p$ and $T\subset \{1,2,\cdots ,n+1\}$.  In other words, $\mathrm{can}|_{A_{01}\oplus A_{11}}=0$ 
and $\mathrm{can}|_{A_{00}\oplus A_{10}}=\mathrm{inc}_{A_{00}\oplus A_{10}}$ where $\textup{inc}_{A_{00}\oplus A_{10}}$ is the canonical inclusion. 
\end{corollary}

We further need to compute the Frobenius map. We proceed by considering the map of algebraic $t$-Bockstein spectral sequences 
\[ 
\adjustbox{scale=.95,center}{\begin{tikzcd}
    \pi_*((\grmot^*\THH(\BP\langle n \rangle ))/(p,v_1,\cdots ,v_{n+1}))[t] \arrow[r,Rightarrow] \arrow[d,"{(-)[\mu^{-p^{n+1}}]}"] &  \pi_*((\grmot^*\TC^{-}(\BP\langle n \rangle ))/(p,v_1,\cdots ,v_{n+1}))\arrow[d,"\varphi"] \\
      \pi_*((\grmot^*\THH(\BP\langle n \rangle )^{tC_p})/(p,v_1,\cdots ,v_{n+1}))[t]  \arrow[r,Rightarrow] &  \pi_*((\grmot^*\TP(\BP\langle n \rangle ))/(p,v_1,\cdots ,v_{n+1})) \,.
    \end{tikzcd}
    }
\]
\begin{proposition}\label{prop:collapse-bpn}
The Frobenius map 
\[ 
    \varphi : \pi_*((\grmot^*\TC^{-}(\BP\langle n\rangle))/(p,v_1,\cdots ,v_{n+1}))\to  \pi_*((\grmot^*\TP(\BP\langle n\rangle))/(p,v_1,\cdots ,v_{n+1}))
\]
is given by $\varphi |_{A_{00}}=\textup{inc}_{A_{00}} $, $\varphi |_{A_{01}}$ is an isomorphism and $\varphi |_{A_{10}\oplus A_{11}}=0$,
where $\textup{inc}_{A_{00}}$ is the canonical inclusion. 
\end{proposition}

\begin{proof}
This follows from \autoref{thm:Hodge-Tateresults} and \autoref{cor:TC-BPn} together with~\cite[Corollary 3.2.6]{AKHW24}, which implies that the classes in $\mathbb{F}_p\langle \lambda_1,\lambda_2,\cdots ,\lambda_{n+1} \rangle$ are equalized by $\mathrm{can}$ and $\varphi$.
\end{proof}
    
We now present our computation of syntomic cohomology of $\BP\langle n \rangle$. 
\begin{thm}\label{thm:syntomic-BPn}
Let $n\ge -1$ and let $\langle n+1\rangle =\{1,2,\cdots ,n+1\}$. The mod $(p,v_1,\dots ,v_{n+1})$-syntomic cohomology of $\BP\langle n\rangle$ is 
\begin{equation}\label{eq:syntomic-bpn-modvn}
\mathbb{F}_p\langle \partial,\lambda_1, \lambda_2,\cdots \lambda_{n+1}\rangle \oplus \bigoplus_{j=1}^{n+1} \mathbb{F}_p\langle \lambda_{s} : s\in \langle n+1\rangle -\{j\} \rangle \{ \Xi_{j,d} : 0 < d < p \} 
\end{equation}
as a $\mathbb{F}_p\langle \partial,\lambda_1,\lambda_2,\cdots ,\lambda_{n+1}\rangle $-module. In particular, it is a finite bigraded free $\bF_p$-module of dimension $2^{n+2}+2^n(n+1)(p-1)$ 
with generators concentrated in degrees $[-1,\sum_{i=1}^{n+1}2p^{i}-n-1]$ and Adams weights $[0,n+2]$. Here the classes are in bidegrees  $\|\partial\|=(-1,1)$, $\|\lambda_i\|=(2p^i-1,1)$,  $\|\Xi_{j,d}\|=(2p^j-1-2dp^{j-1},1)$ for $1\le i,j\le n+1$ and $0<d<p$.
\end{thm}
\begin{proof}
Let $A_{00}$, $A_{10}$, $A_{01}$ and $A_{11}$ be defined as in \eqref{a00}--\eqref{a11}. We determined $\mathrm{can}|_{A_{01}\oplus A_{11}}=0$ 
and $\mathrm{can}|_{A_{00}\oplus A_{10}}=\mathrm{inc}_{A_{00}\oplus A_{10}}$
in \autoref{cor:TC-BPn}. We determined $\varphi |_{A_{00}}=\textup{inc}_{A_{00}} $, $\varphi |_{A_{01}}$ is an isomorphism and $\varphi |_{A_{10}\oplus A_{11}}=0$ in \autoref{prop:collapse-bpn}. We therefore conclude that $\ker(\varphi-\can)=A_{00}\oplus A_{11}$ and $\coker(\varphi-\can)=A_{00}$. We write $\partial$ to indicate 
that classes that are multiples $\partial$ of come from the boundary map in the long exact sequence mod $(p,v_1,\cdots ,v_{n+1})$ induced by the fiber sequence 
\[
\grmot^*\TC(\BP\langle n\rangle)\to\grmot^*\TC^{-}(\BP\langle n\rangle) \overset{\varphi-\can}{\longrightarrow} \grmot^*\TP(\BP\langle n\rangle) \,.
\]
Note that we can use $\mathbb{F}_p\langle \lambda_1,\lambda_2,\cdots ,\lambda_{n+1}\rangle$-module structure to determine that $\pi_*\grmot^*\TC(\BP\langle n\rangle)$ splits as a $\mathbb{F}_p\langle \lambda_1,\lambda_2,\cdots ,\lambda_{n+1}\rangle$-module. 
We present our answer using the notation appearing in~\cite[Theorem~1.1]{AKACHR25}. 
\end{proof}
\begin{corollary}\label{mod-(p ,...,v_n) syntomic}
Let $n\ge -1$ and let $\langle n+1\rangle=\{1,2,\cdots ,n+1\}$. 
The mod $(p,v_1,\dots ,v_{n})$ syntomic cohomology of $\BP\langle n\rangle$ is 
\begin{equation}\label{eq:syntomic-bpn}
\mathbb{F}_{p}[v_{n+1}] \otimes \left (\mathbb{F}_p\langle \partial ,\lambda_1,\lambda_2, \cdots, \lambda_{n+1}\rangle \oplus \bigoplus_{j=1}^{n+1} \mathbb{F}_p\langle \lambda_{s} : \langle n+1\rangle -\{j\} \rangle \{ \Xi_{j,d}: 0 < d < p \} \right ) \,.
\end{equation}
In particular, it is a bigraded free $\bF_p[v_{n+1}]$-module of dimension $2^{n+2}+2^n(n+1)(p-1)$ with generators concentrated in degrees $[-1,\sum_{i=1}^{n+1}2p^{i}-n-1]$ 
and Adams weights $[0,n+2]$. Here the classes are in bidegrees  $\|v_{n+1}\|=(2p^{n+1}-2,0)$, $\|\partial\|=(-1,1)$, $\|\lambda_i\|=(2p^i-1,1)$,  $\|\Xi_{j,d}\|=(2p^j-1-2dp^{j-1},1)$ for $1\le i,j\le n+1$ and $0<d<p$.
\end{corollary}
\begin{proof} 
First, note that the $v_{n+1}$-Bockstein differentials raise Adams weight by exactly $1$ so it suffices to consider one Adams weight at a time. 
We then observe that $\mathbb{F}_p\langle \partial ,\lambda_1,\lambda_2,\cdots ,\lambda_{n+1}\rangle$ is a subring of $\pi_*\TC(\MU)/p$ by \cite[Definition~3.2.4, Lemma~3.2.5.]{AKHW24} so the differentials are $\partial, \lambda_1,\lambda_2,\cdots ,\lambda_{n+1}$-linear. 
It therefore suffices to show that there are no $v_n$-Bockstein differentials on the generators of the $\mathrm{E}_1$-page of the $v_{n+1}$-Bockstein spectral sequence as a $\mathbb{F}_p[v_{n+1}]\langle \partial ,\lambda_1,\lambda_2,\cdots ,\lambda_{n+1}\rangle$-module. 

The only generator in Adams weight $0$ is $1$ and this cannot be the source of a $v_{n+1} $-Bockstein differential because the $\mathbb{F}_p[v_{n+1}]$-module generators of the $\mathrm{E}_1$-page of the $v_{n+1}$-Bockstein spectral sequence are concentrated in degrees $[-1,\sum_{i=1}^{n+1}2p^{i}-n-1]$ by \autoref{thm:syntomic-BPn}. This handles the case $n=-1$. 

Let $n\ge 0$. The generators in Adams weight $1$ are $\Xi_{j,d}$ for $0<d<p$ and $1\le j\le n+1$. The classes in Adams weight $2$ are $\partial \lambda_k$ for $k\ge 1$, $\lambda_i\lambda_j$ for $i\ne j\in \{1,\cdots ,n+1\}$, and $\lambda_s\Xi_{j,d}$ for $s\in \{1,2,\cdots, n+1\}-\{j\}$, $0<j<p$, and $0<d<p$. The class of lowest degree in Adams weight $2$ is $\partial \lambda_1$ in degree $2p-2$ and the class in highest degree in Adams weight $1$, which is a generator as a $\mathbb{F}_p\langle \partial ,\lambda_1,\lambda_2,\cdots ,\lambda_{n+1}\rangle$-module, is $\Xi_{n+1,1}$ in degree $2p^{n+1}-2p^n-1$, so since 
\[ 
 |\partial \lambda_1|+|v_{n+1}|+1=   2p-2+2p^{n+1}-2+1>2p^{n+1}-2p^n-1
\]
for all primes $p$ and all $n\ge 0$. This handles all the $\mathbb{F}_p[v_{n+1}]\langle \partial, \lambda_1,\cdots ,\lambda_{n+1}\rangle$-module generators and therefore the $v_{n+1}$, $\partial$, $\lambda_1$, $\cdots$ ,$\lambda_{n+1}$-linearity proves the claim.  
\end{proof}

 \section{Algebraic K-theory}\label{sec:k-theory}
In this section, we compute the mod $(p,v_1,v_2)$-algebraic K-theory of $\BP\langle 2\rangle$ at primes $p\ge 5$ and the mod $(p,v_1,v_2,v_3)$-algebraic K-theory of $\BP\langle 3\rangle$ at primes $p\ge 7$. 
We then prove the Lichtenbaum---Quillen property for algebraic K-theory of $\BP\langle n\rangle$. 
Finally, we prove that the telescope conjecture holds for the algebraic K-theory of $\BP\langle n\rangle$. A key tool will be the \emph{motivic spectral sequence}
\[
\mathrm{E}_2^{s,2w-s}=\pi_s((\grmot^w \TC(\BP \langle n \rangle))/(p,\cdots ,v_n))\implies \pi_s(\TC(\BP \langle n \rangle)/(p,v_1,\cdots ,v_n)) \,,
\]
which exists when $n=0$, $n=1$ and $p\ge 3$, $n=2$ and $p\ge 5$ and $n=3$ and $p\ge 7$ in light of \autoref{all-defin-same}. This is the spectral sequence associated to the filtered spectrum 
\[\mathrm{fil}_{\mot}^*\TC(\BP \langle n \rangle)_p^{\wedge}\otimes_{\mathrm{fil}_{\textup{ev}}^*\mathbb{S}} \mathrm{fil}_{\textup{ev}/\mathbb{S}}^*(\mathbb{S}/(p,v_1,\cdots,v_n)) \]
and it conditionally converges by~\cite[Corollary~A.2.9]{HRW22}. Recall that $s$ is called the \emph{degree}, $w$ is called the \emph{weight} and $2w-s$ is called the Adams weight. With the choice of coordinates $(s,2w-s)$ given by (degree, Adams weight), the differential convention of the spectral sequence is the standard Adams convention
\[ 
	d_r:\mathrm{E}_2^{s,2w-s} \to \mathrm{E}_2^{s-1,2w-s+r}\,.
\]

\subsection{Algebraic K-theory at low heights}
We provide some applications of our work to computations of algebraic K-theory of $\BP\langle 2\rangle$ at primes $p\ge 5$ and $\BP\langle 3\rangle$ at primes $p\ge 7$.

\begin{thm}\label{TCBP2}
Let $p\ge 5$. Then $\pi_{*}(\TC(\BP \langle 2 \rangle)/(p,v_1,v_2))$ is isomorphic to 
\[ 
\bF_{p}[v_{3}]\langle  \partial ,\lambda_1, \lambda_{2}, \lambda_{3}\rangle\oplus \bigoplus_{j=1}^{3} \mathbb{F}_p[v_3]\langle \lambda_{s} : 1\le s\le 3, s\ne j
\rangle \{ \Xi_{j,d} : 0 < d < p\} 
\]
as $\mathbb{F}_p[v_3^d]\langle  \partial ,\lambda_1, \lambda_{2}, \lambda_{3}\rangle$-modules for some integer $d\ge 1$. Here $|v_3|=2p^{3}-2$, $|\partial|=-1$, $|\lambda_i|=2p^i-1$ and $|\Xi_{j,d}|=2p^j-1-2dp^{j-1}$ for all $1\le i,j\le n+1$ and $0<d<p$. 
\end{thm}
\begin{proof}
For simplicity, we display the mod $(2,v_1,v_2,v_3)$-syntomic cohomology of $\BP\langle 2\rangle$ in \autoref{fig:syntomicBP2}, 
but the reader should be able to extrapolate the necessary information from this picture to verify the statements in this proof. We will refer to classes in Adams weight $m$ as classes on the $m$-line, since this is how we depict the spectral sequence, see \autoref{fig:syntomicBP2} for example. 
Since the motivic spectral sequence is concentrated on lines $0$, $1$, $2$, $3$ and $4$ the only possible motivic differential is one 
of length $3$ from $0$-line to the $3$-line or the $1$-line to $4$-line. For differentials from the $0$-line to the $3$-line, the only possible differentials 
have source $v_{3}^{k}$, but the targets are bidegree $((2p^{3}-2)k,3)$ and the mod $(p,v_1,v_2)$-syntomic cohomology of $\BP\langle 2\rangle$ 
is trivial in these bidegrees. Similarly, the only possible nontrivial targets of a differential of length $3$ from the $1$-line to the $4$-line are 
$\lambda_{1}\lambda_{2}\lambda_{3}v_{3}^{k}$, but there is a gap between $(2p^{3}+2p^{2}+2p-1,1)$ and $(2p^{3}+4p^{2}+2p-1,1)$ 
so there could not be any differential hitting $\delta \lambda_{1}\lambda_{2}\lambda_{3}$. This gap persists in positive dimensions 
mod $2p^{3}-2$, so there are no possible differentials hitting $\lambda_{1}\lambda_{2}\lambda_{3}v_{3}^{k}$. 
Since we can view $\TC(\BP\langle 2\rangle)/(p,v_1,v_2)$ as a  $\textup{End}(\mathbb{S}/(p,v_1,v_2))$-module, 
which has a $v_3^d$-self map for some $d\ge 1$ and $\pi_0\textup{End}(\mathbb{S}/(p,v_1,v_2))=\mathbb{F}_p$, 
this identification is an identification of $\mathbb{F}_p[v_3^d]$-modules. 
Here, we write $\Xi_{j,d}$ for the class detected by $\Xi_{j,d}$ without indeterminacy in the motivic spectral sequence. 
\end{proof}

\begin{cor}\label{Hurewicz-BP2}
The classes $\alpha_1$, $\beta_1'$, and $\gamma_1^{''}$ in $V(2)_*$ map under the unit map 
\[ \mathbb{S}\to \mathrm{TC}(\mathrm{BP}\langle 2\rangle) \]
to the classes $\Xi_{1,1}$, $\Xi_{2,1}$ and $\Xi_{3.1}$, respectively. 
\end{cor}
\begin{proof}
The classes $\Xi_{1,1}$, $\Xi_{2,1}$ and $\Xi_{3.1}$, respectively are in the image of $t_1$, $t_1^p$, and $t_1^{p^2}$ respectively under the unit map 
\[ \pi_*\gr_{\ev}^*\mathbb{S}\to \pi_*\grmot\mathrm{TC}(\mathrm{BP}\langle 2\rangle)\] 
since  $\Xi_{1,1}=t\sigma^2t_1$, $\Xi_{2,1}=t^p(\sigma^2t_1)^p$ and $\Xi_{3,1}=t^{p^2}(\sigma^2t_1)^{p^2}$ and the classes  $t_1$, $t_1^p$, and $t_1^{p^2}$ detect $\alpha_1$, $\beta_1'$, and $\gamma_1^{''}$ in $V(2)_*$ respectively. There are no other classes in higher or lower motivic filtration in the same degree. 
\end{proof}

\begin{cor}\label{KBP2}
Let $p\ge 5$. There is a long exact sequence 
\[ 
0 \to \Sigma^{-2}M_2\to \pi_*(\K(\BP\langle 2\rangle)/(p,v_1,v_2))\to \pi_*(\TC(\BP\langle 2\rangle)/(p,v_1,v_2))\to \Sigma^{-1}\bF_p\{\partial\}\to 0
\]
where $M_2=\mathbb{F}_p\{\overline{\tau}_1,\overline{\tau}_2,\overline{\tau}_1\overline{\tau}_2\}$ and $|\overline{\tau}_i|=2p^i-1$ for $i=1,2$
and $\pi_*(\mathrm{K}(\BP\langle 2\rangle)/(p,v_1,v_2))[v_3^{-1}]$ 
is isomorphic to 
\[ 
\bF_{p}[v_{3}^{\pm 1}]\otimes \bigl( \mathbb{F}_p\langle  \partial ,\lambda_1, \lambda_{2}, \lambda_{3}\rangle\oplus \bigoplus_{j=1}^{3} \mathbb{F}_p\langle \lambda_{s} : 1\le s\le 3, s\ne j
\rangle \{ \Xi_{j,d} : 0 < d < p\} \bigr)
\]
as $\mathbb{F}_p[v_3^{\pm \ell}]\langle  \partial ,\lambda_1, \lambda_{2}, \lambda_{3}\rangle$-modules for some integer $\ell\ge 1$.
Here $|v_3|=2p^{3}-2$, $|\partial|=-1$, $|\lambda_i|=2p^i-1$ and $|\Xi_{j,d}|=2p^j-1-2dp^{j-1}$ for all $1\le i,j\le 3$ and $0<d<p$. 
\end{cor}
\begin{proof}
The same strategy as~\cite[Theorem 12.20]{AKACHR25} applies. By~\cite[Theorem~3.1.14]{DGM13}, there is a fiber sequence
\[ 
\begin{tikzcd}
\mathrm{K}(\BP\langle 2\rangle)/(p,v_1,v_2)\arrow{r} &  \mathrm{TC}(\BP\langle 2\rangle)/(p,v_1,v_2)\arrow{r}{\omega} & \Sigma^{-1}\mathbb{Z}/(p,v_1,v_2)
\end{tikzcd}
\]
where $\pi_*(\mathbb{Z}/(p,v_1,v_2))=\mathbb{F}_p\langle \overline{\tau}_1, \overline{\tau}_2\rangle$. 
This already implies the second claim. 
We know that the classes $\partial \lambda_1$, $\partial \lambda_2$ and $\partial \lambda_3$ lift to mod $p$ homotopy classes so by the commuting diagram
\[
\begin{tikzcd}
\mathrm{TC}(\BP\langle 2\rangle)/(p)\ar[r] \ar[d] & \Sigma^{-1}\mathbb{Z}/p=\Sigma^{-1}\mathbb{F}_p \ar[d] \\ 
\mathrm{TC}(\BP\langle 2\rangle)/(p,v_1,v_2) \ar[r] & \Sigma^{-1}\mathbb{Z}/(p,v_1,v_2)
\end{tikzcd}
\]
they have trivial image in $\pi_* \Sigma^{-1}\mathbb{Z}/(p,v_1,v_2)$. 
The classes $\lambda_1$ is in the image of the map
\[\pi_*(\mathrm{K}(\BP)/p) \to \pi_*(\mathrm{TC}(\BP\langle 2\rangle)/(p,v_1,v_2))\] 
by construction and the class $\Xi_{2,1}$ is in the image of the map
\[ \pi_*\mathbb{S}/(p,v_1,v_2) \to \pi_*\TC(\BP\langle 2\rangle/(p,v_1,v_2)\] 
so their product has trivial image in  $\pi_* \Sigma^{-1}\mathbb{Z}/(p,v_1,v_2)$. Therefore, the map 
\[\pi_*\mathrm{TC}(\BP\langle 2\rangle)/(p,v_1,v_2) \longrightarrow \pi_*\Sigma^{-1}\mathbb{Z}/(p,v_1,v_2) \]
 is the zero map in non-negative degrees. Since mod $(p,v_1,v_2)$-algebraic K-theory of $\BP\langle n\rangle$ is connective, this implies the first claim. 
\end{proof}

\begin{remark}
This extends work of the author with Ausoni, Culver, H\"oning, and Rognes \cite[Theorem~1.2]{AKACHR25} 
to all $\bE_{1}$-$\MU$-algebra forms of $\BP \langle 2 \rangle$. It also extends the computation to the prime $p=5$.
\end{remark}

\begin{figure}[!ht]
\resizebox{\textwidth}{!}{ 
\begin{tikzpicture}[radius=1,xscale=1.3,yscale=2]
\foreach \n in {-2,0,...,60} \node [below] at (\n,-.8-1) {$\n$};
\foreach \s in {-1,0,...,10} \node [left] at (-.3-2,\s) {$\s$};
\draw [thin,color=lightgray] (-2,-1) grid (60,10);
\node [below] at (0,0) {$1$};

\node [below] at (-1,1) {$\partial$};
\node [below] at (1,1) {$\Xi_{1,1}$};
\node [above] at (3,1) {$\lambda_1$};
\node [below] at (3,1) {$\Xi_{2,1}$};
\node [above] at (7,1) {$\lambda_2$};
\node [below] at (7,1) {$\Xi_{3,1}$};
\node [below] at (15,1) {$\lambda_3$};
\node [above] at (15,1) {$\Xi_{4,1}$};
\node [below] at (31,1) {$\lambda_4$};

\node [below] at (2,2) {$\partial\lambda_1$};
\node [above] at (6,2) {$\partial\lambda_2$};
\node [below] at (6,2) {$\lambda_1\Xi_{2,1}$};
\node [below] at (8,2) {$\Xi_{1,1}\lambda_2$};
\node [above] at (10,2) {$\lambda_1\lambda_2$};
\node [below] at (10,2) {$\lambda_1\Xi_{3,1}$};
\node [below] at (14,2) {$\partial\lambda_3$};
\node [above] at (14,2) {$\lambda_2\Xi_{3,1}$};
\node [above] at (16,2) {$\Xi_{1,1}\lambda_3$};
\node [below] at (18,2) {$\lambda_1\lambda_3$};
\node [above] at (18,2.2) {$\Xi_{2,1}\lambda_3$};
\node [above] at (18,2) {$\Xi_{4,1}\lambda_1$};
\node [below] at (22,2) {$\lambda_2\lambda_3$};
\node [above] at (22,2) {$\Xi_{4,1}\lambda_2$};
\node [below] at (30,2) {$\partial \lambda_4$};
\node [above] at (30,2) {$\Xi_{4,1}\lambda_3$};
\node [below] at (32,2) {$\Xi_{1,1}\lambda_4$};
\node [below] at (34,2) {$\lambda_1\lambda_4$};
\node [above] at (34,2) {$\Xi_{2,1}\lambda_4$};
\node [below] at (38,2) {$\lambda_2\lambda_4$};
\node [above] at (38,2) {$\Xi_{3,1}\lambda_4$};
\node [below] at (46,2) {$\lambda_3\lambda_4$};

\node [above] at (9,3) {$\partial\lambda_1\lambda_2$};
\node [above] at (17,3) {$\partial\lambda_1\lambda_3$};
\node [below] at (17,3) {$\lambda_1\lambda_2\Xi_{3,1}$};
\node [above] at (21,3) {$\lambda_1\Xi_{2,1}\lambda_3$};
\node [below] at (21,3) {$\partial\lambda_2\lambda_3$};
\node [above] at (23,3) {$\Xi_{1,1}\lambda_2\lambda_3$};
\node [above] at (25,3) {$\lambda_1\lambda_2\lambda_3$};
\node [below] at (25,3) {$\Xi_{4,1}\lambda_1\lambda_2$};
\node [below] at (33,3) {$\partial\lambda_1\lambda_4$};
\node [above] at (33,3) {$\Xi_{4,1}\lambda_1\lambda_3$};
\node [above] at (37,3) {$\Xi_{4,1}\lambda_2\lambda_3$};
\node [below] at (37,3) {$\partial \lambda_2\lambda_4$};
\node [below] at (37,2.8) {$\Xi_{2,1} \lambda_1\lambda_4$};
\node [below] at (39,3) {$\Xi_{1,1}\lambda_2\lambda_4$};
\node [below] at (41,3) {$\lambda_1\lambda_2\lambda_4$};
\node [above] at (41,3) {$\Xi_{3,1}\lambda_1\lambda_4$};
\node [above] at (45,3) {$\Xi_{3,1}\lambda_2\lambda_4$};
\node [below] at (45,3) {$\partial\lambda_3\lambda_4$};
\node [below] at (47,3) {$\Xi_{1,1}\lambda_3\lambda_4$};
\node [below] at (49,3) {$\lambda_1\lambda_3\lambda_4$};
\node [above] at (49,3) {$\Xi_{2,1}\lambda_3\lambda_4$};
\node [above] at (53,3) {$\lambda_2\lambda_3\lambda_4$};

\node [above] at (24,4) {$\partial\lambda_1\lambda_2\lambda_3$};
\node [below] at (40,4) {$\partial\lambda_1\lambda_2\lambda_4$};
\node [above] at (40,4) {$\Xi_{4,1}\lambda_1\lambda_2\lambda_3$};
\node [above] at (48,4) {$\Xi_{3,1}\lambda_1\lambda_2\lambda_4$};
\node [below] at (48,4) {$\partial\lambda_1\lambda_3\lambda_4$};
\node [above] at (52,4) {$\Xi_{2,1}\lambda_1\lambda_3\lambda_4$};
\node [below] at (52,4) {$\partial\lambda_2\lambda_3\lambda_4$};
\node [above] at (54,4) {$\Xi_{1,1}\lambda_2\lambda_3\lambda_4$};
\node [above] at (56,4) {$\lambda_1\lambda_2\lambda_3\lambda_4$};

\node [above] at (55,5) {$\partial\lambda_1\lambda_2\lambda_3\lambda_4$};

\end{tikzpicture}
}
\caption{The mod $(2,v_1,v_2,v_3,v_4)$-syntomic cohomology of $\BP\langle 3\rangle$. 
\label{fig:syntomicBP3}
}
\end{figure}

\begin{thm}\label{TCBP3}
Let $p\ge 7$. Then $\pi_{*}(\TC(\BP \langle 3 \rangle)/(p,v_1,v_2,v_3))$ is isomorphic to 
\[ 
\bF_{p}[v_{4}]\langle  \partial ,\lambda_1, \lambda_{2}, \lambda_{3},\lambda_4\rangle\oplus \bigoplus_{j=1}^{4} \mathbb{F}_p[v_3]\langle \lambda_{s} : 1\le s\le 4, s\ne j
\rangle \{ \Xi_{j,d} : 0 < d < p\} 
\]
as $\mathbb{F}_p[v_4^d]\langle  \partial ,\lambda_1, \lambda_{2}, \lambda_{3},\lambda_4\rangle$-modules for some integer $d\ge 1$. Here $|v_4|=2p^{4}-2$, $|\partial|=-1$, $|\lambda_i|=2p^i-1$ and $|\Xi_{j,d}|=2p^j-1-2dp^{j-1}$ for all $1\le i,j\le n+1$ and $0<d<p$. 
\end{thm}

\begin{proof}
For simplicity, we display the mod $(2,v_1,v_2,v_3,v_4)$-syntomic cohomology of $\mathrm{BP}\langle 3\rangle $ in
\autoref{fig:syntomicBP3}, but the reader should be able to extrapolate the necessary information from this
picture to verify the statements in this proof. Recall that the $d_r$-differentials differentials decrease degree by one and increase motivic filtration by $r$ and by definition the even lines are concentrated in even degrees and the odd lines are concentrated in odd degrees so we can only have motivic differentials of odd length. 
Since the motivic spectral sequence is concentrated on lines $0$, $1$, $2$, $3$ and $4$ the only possible motivic differentials are differentials of length $5$ from the $0$-line to the $5$-line 
and differentials of length $3$ from the $0$-line to the $3$-line, the $1$-line to the $4$-line, and the $2$-line to the $5$-line.  

For differentials of length $5$, the only possible sources of nontrivial differentials are differentials on $v_4^k$ in degree $(2p^4-2)k$. However, $|\partial \lambda_1\lambda_2\lambda_3\lambda_4v_4^j|=(2p^4-2)(j+1)+2p^3+2p^2+2p-3$ so for all primes, the classes $\partial \lambda_1\lambda_2\lambda_3\lambda_4v_4^j$ are not in degree one less than $|v_4^k|=(2p^4-2)k$ for any integers $k,j\ge 0$. 

For differentials of length $3$ with source on the zero line, again a nontrivial differential would have to have source $v_4^k$ in degree $(2p^4-2)k$.  It suffices to check that there are no classes in degree $-1\mod 2p^4-2$ on the three line. The $3$-line is generated by two families of classes. 
First, there are classes $\partial \lambda_1\lambda_2$,   $\partial \lambda_1\lambda_3$,  $\partial \lambda_1\lambda_4$,  $\partial \lambda_2\lambda_3$,  $\partial \lambda_2\lambda_4$,  $\partial \lambda_3\lambda_4$, $\lambda_1\lambda_2\lambda_3$, $\lambda_1\lambda_2\lambda_4$, $\lambda_1\lambda_3\lambda_4$, $\lambda_2\lambda_3\lambda_4$ in degrees $2p^2+2p-3$, $2p^3+2p-3$, $2p-1$, $2p^3+2p^2-3$, $2p^2-1$, $2p^3-1$, 
$2p^3+2p^2+2p-3$, $2p^2+2p-1$, $2p^3+2p-1$ and $2p^3+2p^2-1$ mod $2p^4-2$ respectively. 
Second, there are classes
$\Xi_{1,d} \lambda_2\lambda_3$, $\Xi_{2,d}\lambda_1\lambda_3$, $\Xi_{3,d}\lambda_1\lambda_2$, $\Xi_{1,d}\lambda_2\lambda_4$,  $\Xi_{2,d}\lambda_1\lambda_4$,  $\Xi_{4,d}\lambda_1\lambda_2$, 
$\Xi_{1,d}\lambda_3\lambda_4$,$\Xi_{3,d}\lambda_1\lambda_4$, $\Xi_{4,d}\lambda_1\lambda_3\lambda_4$, $\Xi_{2,d}\lambda_3\lambda_4$, $\Xi_{3,d}\lambda_2\lambda_4$ and 
$\Xi_{4,d}\lambda_2\lambda_3$ 
in degrees 
$2p^3+2p^2+2p-3-2d$, $2p^3+2p^2+2p-3-2dp$, $2p^3+2p^2+2p-3-2dp^2$,
$2p^2+2p-1-2d$, $2p^2+2p-1-2dp$, $2p^4+2p^2+2p-3-2dp^3$,
$2p^3+2p-1-2d$, $2p^2+2p-1-2dp^2$, $2p^2+2p-1-2dp^3$, $2p^3+2p^2-1-2dp$, $2p^3+2p^2-1-2dp^2$ and $2p^3+2p^2-1-2dp^3$ mod $2p^4-2$ respectively for $0<d<p$. Consequently, there is no possible nontrivial differential from the $0$-line to the $3$-line for degree reasons. 

For differentials of length $3$ with source on the $1$-line, the possible classes that can be sources of nontrivial differentials are of the form $\partial$, $\lambda_1$, $\lambda_2$, $\lambda_3$, $\lambda_4$, 
$\Xi_{1,d}$, $\Xi_{2,d}$, $\Xi_{3,d}$ and $\Xi_{4,d}$ in degrees $-1$, $2p-1$, $2p^2-1$, $2p^3-1$, $1$, $2p-1-2d$, $2p^2-1-2pd$, $2p^3-2p^2d$ and $2p^4-1-2dp^3$ $\mod 2p^4-2$ for $0<d<p$ respectively. 
The potential nonzero targets are $\partial \lambda_1\lambda_2\lambda_3$, $\partial\lambda_1\lambda_2\lambda_4$, $\partial \lambda_1\lambda_3\lambda_4$, $\partial \lambda_2\lambda_3\lambda_4$, $\Xi_{1,d}\lambda_2\lambda_3\lambda_4$, $\Xi_{2,d}\lambda_1\lambda_3\lambda_4$, $\Xi_{3,d}\lambda_1\lambda_2\lambda_4$ and $\Xi_{4,d}\lambda_2\lambda_3$ in degrees $2p^3+2p^2+2p-4$, $2p^2+2p-2$, $2p^3+2p-2$, $2p^3+2p^2-2$, $2p^3+2p^2+2p-2-2d$, $2p^3+2p^2+2p-2-2pd$, $2p^3+2p^2+2p-2-2p^2d$ and $2p^3+2p^2+2p-2-2p^3d$ $\mod$ $2p^4-2$ for $0<d<p$ and $0\le j\le 3$. Since none of the potential targets are in degree one less then the potential sources mod $2p^{n+1}-2$, there are no possible such nontrivial $d_3$-differentials. 

For differentials of length $3$ with source on the $2$-line, the possible classes that can be sources of nontrivial differentials are of the form $\partial \lambda_1$, $\partial \lambda_2$, $\partial \lambda_3$, $\partial \lambda_4$, $\lambda_1\lambda_2$, $\lambda_1\lambda_3$, $\lambda_1\lambda_4$, $\lambda_2\lambda_3$, $\lambda_2\lambda_4$, $\lambda_3\lambda_4$, $\Xi_{1,d}\lambda_2$, $\Xi_{1,d}\lambda_3$, $\Xi_{1,d}\lambda_4$, $\Xi_{2,d}\lambda_3$, $\Xi_{2,d}\lambda_4$, $\Xi_{3,d}\lambda_4$, $\Xi_{2,d}\lambda_1$, $\Xi_{3,d}\lambda_1$, $\Xi_{4,d}\lambda_1$, $\Xi_{3,d}\lambda_2$, $\Xi_{4,d} \lambda_2$, $\Xi_{4,d}\lambda_3$ in degrees $2p-2$, $2p^2-2$, $2p^3-2$, $0$, $2p^2+2p-2$, $2p^3+2p-2$, $2p$, $2p^3+2p^2-2$, $2p^2$, $2p^3$, $2p^2+2p-2-2d$, $2p^3+2p-2-2d$, $2p-2d$, $2p^3+2p^2-2-2pd$, $2p^2-2dp$, $2p^3-2dp^2$, $2p^2+2p-2-2dp$, $2p^3+2p-2dp^2$, $2p-2dp^3$, $2p^3+2p^2-2-2dp^2$, $2p^2-2dp^3$ and $2p^3-2dp^3$ $\mod$ $2p^4-2$ respectively. However, since $|\partial \lambda_1\lambda_2\lambda_3\lambda_4|=2p^3+2p^2+2p-3$ $\mod $ $2p^4-2$, there are no possible such differentials of length $3$. 
\end{proof}

\begin{cor}\label{KBP3}
Let $p\ge 7$. There is a long exact sequence
\[ 
0 \to \Sigma^{-2}M_3\to \pi_*(\K(\BP\langle 3\rangle)/(p,v_1,v_2,v_3))\to \pi_*(\TC(\BP\langle 3\rangle)/(p,v_1,v_2,v_3))\to \Sigma^{-1}\bF_p\{\partial\}\to 0
\]
where $M_3=\mathbb{F}_p\{\overline{\tau}_1,\overline{\tau}_2,\overline{\tau}_3,\overline{\tau}_1\overline{\tau}_2,\overline{\tau}_1\overline{\tau}_3,\overline{\tau}_2\overline{\tau}_3,\overline{\tau}_1\overline{\tau}_2\overline{\tau}_3 \}$ and $|\overline{\tau}_i|=2p^i-1$ for $i=1,2,3$. Moreover, we compute that $\pi_*(\mathrm{K}(\BP\langle 3\rangle)/(p,v_1,v_2,v_3))[v_4^{-1}]$ 
is isomorphic to 
\[ 
\bF_{p}[v_{4}^{\pm 1}]\otimes \left ( \mathbb{F}_p\langle  \partial ,\lambda_1, \lambda_{2}, \lambda_{3}, \lambda_{4} \rangle\oplus \bigoplus_{j=1}^{4} \mathbb{F}_p\langle \lambda_{s} : 1\le s\le 3, s\ne j
\rangle \{ \Xi_{j,d} : 0 < d < p\} \right )
\]
as $\mathbb{F}_p[v_4^{\pm \ell}]\langle  \partial ,\lambda_1, \lambda_{2}, \lambda_{3},\lambda_{4}\rangle$-modules for some integer $\ell\ge 1$.
Here $|v_4|=2p^{4}-2$, $|\partial|=-1$, $|\lambda_i|=2p^i-1$ and $|\Xi_{j,d}|=2p^j-1-2dp^{j-1}$ for all $1\le i,j\le 4$ and $0<d<p$. 
\end{cor}

\begin{proof}
By~\cite[Theorem~3.1.14]{DGM13}, there is a fiber sequence
\[ 
\begin{tikzcd}
\mathrm{K}(\BP\langle 3\rangle)/(p,v_1,v_2,v_3)\arrow{r} & \mathrm{TC}(\BP\langle 3\rangle)/(p,v_1,v_2,v_3) \arrow{r}{\omega} & \Sigma^{-1}\mathbb{Z}/(p,v_1,v_2,v_3)
\end{tikzcd}
\]
where $\pi_*\mathbb{Z}/(p,v_1,v_2,v_3)=\mathbb{F}_p\langle \overline{\tau}_1, \overline{\tau}_2, \overline{\tau}_3\rangle$. This already implies the second claim. 
We know that the classes $\partial \lambda_1$, $\partial \lambda_2$, $\partial \lambda_3$, $\partial \lambda_1\lambda_2$, $\partial \lambda_1\lambda_3$, $\partial \lambda_2\lambda_3$ and $\partial \lambda_1\lambda_2\lambda_3$ lift to mod $p$ homotopy classes so by the commuting diagram
\[
\begin{tikzcd}
\mathrm{TC}(\BP\langle 3\rangle)/(p)\ar[r] \ar[d] & \Sigma^{-1}\mathbb{Z}/p=\Sigma^{-1}\mathbb{F}_p \ar[d] \\ 
\mathrm{TC}(\BP\langle 3\rangle)/(p,v_1,v_2,v_3) \ar[r] & \Sigma^{-1}\mathbb{Z}/(p,v_1,v_2,v_3)
\end{tikzcd}
\]
these classes have trivial image under the map $\omega_*$. 
Moreover, by construction and by \autoref{KBP3} the classes $\lambda_1$, $\lambda_2$, $\lambda_3$, and $\lambda_1\lambda_2$, $\lambda_1\lambda_3$, lift to $\pi_*(\TC(\BP)/p)$. Since $\Xi_{2,1}$ and $\Xi_{3,1}$ are in the image of $\pi_*(\mathbb{S}/(p,v_1,v_2,v_3))\to \pi_*(\mathrm{TC}(\BP\langle 3\rangle)/(p,v_1,v_2,v_3))$ by \autoref{Hurewicz-BP2}, so the products $\lambda_1\Xi_{2,1}$, $\lambda_2\Xi_{3,1}$ and $\lambda_1\lambda_2\Xi_{3,1}$ are not in the image of the map $\omega_*$. Consequently, the map  $\omega_*$ is zero in non-negative degrees. 
Since mod $(p,v_1,v_2,v_3)$-algebraic K-theory of $\BP\langle 3\rangle$ is connective, this implies the first claim. 
\end{proof}

\subsection{Lichtenbaum--Quillen}
In this section, we note that our computations imply various Lichtenbaum--Quillen type properties. 

\begin{thm}\label{thm:finite-TC}
For any integers $i_0,i_1,\cdots ,i_{n+1}$ such that the generalized Smith--Toda complex  $\mathbb{S}/(p^{i_0},v_1^{i_1},\cdots,v_{n+1}^{i_{n+1}})$ exists, the graded abelian group
\[ \pi_*\TC(\BP\langle n\rangle )/(p^{i_0},v_1^{i_1},\cdots, v_{n+1}^{i_{n+1}})\]
is finite; i.e. it is finite in each degree and only non-trivial in finitely many degrees. 
Consequently, the algebraic K-theory $\mathrm{K}(\mathrm{BP}\langle n\rangle)$ of $\mathrm{BP}\langle n\rangle$ has fp-type $n+1$. 
\end{thm}
\begin{proof}
By \autoref{thm:syntomic-BPn}, the bigraded abelian group
$
\pi_*((\grmot^*\TC(\BP\langle n\rangle))/(p,v_1,\cdots,v_{n+1}))
$
is finite in the sense it is finite in all bidegrees and only non-trivial in finitely many bidegrees. 
This immediately implies that 
$
\pi_*((\grmot^*\TC(\BP\langle n\rangle))/(p^{i_0},v_1^{i_1},\cdots,v_{n+1}^{i_{n+1}}))
$
is finite for any integers $i_0,i_1,\cdots, i_{n+1}$ using the finite $p,v_1,\cdots ,v_{n+1}$-Bockstein filtrations. 
Choosing integers $i_0,i_1,\cdots,i_{n+1}$ such that the Moore spectrum $\mathbb{S}/(p^{i_0},v_1^{i_1},\cdots,v_{n+1}^{i_{n+1}})$ exists, 
this immediately implies that 
$
\pi_*(\TC(R)/(p^{i_0},v_1^{i_1},\cdots,v_{n+1}^{i_{n+1}}))
$
is finite by motivic spectral sequence and consequently $\TC(R)$ has fp-type $n+1$ in the sense of \cite{MR99}. 
By~\cite[Theorem~3.1.14]{DGM13}, there is a fiber sequence 
\[
\K(\BP\langle n\rangle )/(p^{i_0},v_1^{i_1},\cdots ,v_{n+1}^{i_{n+1}})\to \TC(\BP\langle n\rangle )/(p^{i_0},v_1^{i_1},\cdots ,v_{n+1}^{i_{n+1}})\to \mathbb{Z}/(p^{i_0},v_1^{i_1},\cdots ,v_{n+1}^{i_{n+1}})
\]
and the algebraic K-theory of $R$ has fp-type $n+1$ since $\pi_*\mathbb{Z}/(p^{i_0},v_1^{i_1},\cdots ,v_{n+1}^{i_{n+1}})$ is also finite. 
\end{proof}

\begin{remark}
This extends~\cite[Theorem B]{HW22} to arbitrary $\mathbb{E}_1$ $\MU$-algebra forms of $\BP\langle n\rangle$. 
\end{remark}
 
This implies the higher height analogue of the Lichtenbaum--Qullen conjecture as proposed by Ausoni and Rognes in~\cite{AR08} as part of the family of conjectures known as the redshift conjectures.
\begin{cor}\label{cor:LQC-K-theory}
Let $n\ge -1$ and $p$ be a prime. The fiber of the map
\[ \K(\BP\langle n\rangle )_{(p)}\longrightarrow  L_{n+1}^{f}\K(\BP\langle n\rangle )_{(p)} 
\]
is bounded above. 
\end{cor}

\begin{proof}
This follows from \autoref{thm:finite-TC} by \cite[Theorem~8.2(2)]{MR99} as observed in \cite[Theorem~3.1.3]{HW22} together with an arithmetic fracture square argument. 
\end{proof}

\subsection{Telescope}
Since it is now known that the telescope conjecture is false in general~\cite{BHLS23}, it becomes an interesting question to ask for which spectra $X$ is the localization map 
\[ L_{n+1}^fX\to L_{n+1}X\] 
is an equivalence. For example, it was conjectured by~\cite[Conjecture~3.9]{MR99} that the telescope conjecture holds for fp-spectra. 
We show that our results imply the telescope conjecture for the algebraic K-theory of $\BP\langle n \rangle$, giving some support for this conjecture in light of \autoref{thm:finite-TC}, which implies that the algebraic K-theory of $\BP\langle n\rangle$ is an fp-spectrum. 
\begin{cor}\label{cor:telescope}
Let $n\ge -1$. The localization map 
\[ L_{n+1}^f\K(\BP\langle n\rangle) \to L_{n+1}\K(\BP\langle n\rangle) \]
is an equivalence. 
\end{cor}
\begin{proof}
Consider the eff map of cyclotomic $\mathbb{E}_\infty$-rings $\THH(\MU)\to \THH(\MU/\MW)$. It is clear that $\pi_*\BP\langle n\rangle_p^{\wedge}$ is a finitely generated $\bZ_p$-module and that $R$ has height $n$ in either case. 
We computed that $\BP\langle n\rangle_p^{\wedge}$ has bounded below topological Hochschild homology in \autoref{prop:Hocschild-MaySSBPn}, the $\THH(\MU)$-module $\THH(\BP\langle n\rangle)$ 
has bounded $\MU$-based motivic cohomological dimension in \autoref{prop:thh-bpn} and that 
\[\THH(\BP\langle n\rangle/\MU)=\THH(R)\otimes_{\THH(\MU)}\MU\] 
is even in \autoref{algebraic eff map} and consequently 
\[\THH(\BP\langle n\rangle/\MU)=\THH(R)\otimes_{\THH(\MU)}\THH(\MU/\MW)\] 
is even. Therefore, the result follows from \cite[Proposition~4.3.2]{AKHW24}. 
\end{proof}

\subsection{Redshift}
We now prove redshift for algebraic K-theory of $\BP\langle n\rangle$. 

\begin{defin}\label{def-height}
We say a spectrum $X$ has \emph{height $n+1$} if $T(n+1)_*X\ne 0$ and $T(m)_*X=0$ for all $m>n+1$. 
\end{defin}
We say a $\mathbb{E}_1$-ring satisfies redshift if it is height $n$ and its algebraic K-theory has height $n+1$. By~\cite{BSY22}, we know for all $\mathbb{E}_\infty$-rings that if $R$ has height $n$, then $\mathrm{K}(R)$ has height $n+1$. 
It remains an interesting question of for which $\mathbb{E}_1$-rings satisfy redshift, see for example~\cite{AKHW24}. We provide a new family of examples. Note that $\BP\langle n\rangle$ has height $n$. 
\begin{cor}\label{cor:redshift}
The algebraic K-theory of $\BP\langle n\rangle$ has height $n+1$. 
\end{cor}
\begin{proof}
By \cite[Corollary~4.12]{LMMT20}, it suffices to prove that the algebraic K-theory of $\BP\langle n\rangle$ has height at least $n+1$. 
Since $\bE_{0}$ algebra structure on $\TC^{-}(\BP\langle n \rangle)$ factors  
\[
\bS \to \K(\BP\langle n\rangle )\to \TC(\BP\langle n\rangle)\to \TC^{-}(\BP\langle n\rangle/\MU)
\]
it suffices to prove that a class in motivic filtration $0$ of $\TC(\BP\langle n\rangle)$ in the kernel of $\can-\varphi$ is not $(p^{i_{0}},\cdots ,v_{n}^{i_{n}})$-torsion for 
any postive integers $i_{0},\cdots ,i_{n}$. and that this class is $v_{n+1}$-periodic. 
Choosing the class $1\in \pi_{0}\TC(\BP \langle n\rangle)$ in the image of the element $1\in \pi_{0}\bS$ then we can conclude that this class has 
both of these properties by \autoref{mod-(p ,...,v_n) syntomic}. 
\end{proof}
This extends~\cite[Corollary~5.0.2]{HW22} to arbitrary $\mathbb{E}_1$-$\MU$-algebra forms of $\BP\langle n\rangle$.

\bibliographystyle{alpha}
\bibliography{kbpn}
\end{document}